\documentclass[smallextended,envcountsect,]{svjour3}
\smartqed
\usepackage{graphicx}
\usepackage{subcaption}

\usepackage{amsfonts}       
\graphicspath{{./}}     
\usepackage{amsmath}
\usepackage{algorithm,algorithmic}
\usepackage{physics}
\usepackage{amssymb} 
\usepackage{bm}
\usepackage{mathtools}
\usepackage{xcolor}

\def\pd{\partial}

\def\R{\mathbb{R}}

\def\dom{\mathrm{dom}}
\def\<{\langle}
\def\>{\rangle}
\def\ker{\mathrm{Kernel}}
\def\Im{\mathrm{Range}}
\def\Mt{\widetilde{M}}

\def\Bt{\widetilde{B}}
\def\Ct{\widetilde{C}}
\def\Ht{\widetilde{H}}

\def\F{\mathcal{F}}
\def\N{\mathcal{N}}
\graphicspath{{./}}

\def\BLambda{\boldsymbol{\Lambda}}

\DeclareMathOperator*{\argmax}{argmax}
\DeclareMathOperator*{\argmin}{argmin}
 \DeclareMathOperator{\prox}{Prox}
  \DeclareMathOperator{\refl}{R}
\def\iden{\mathbb I}

\journalname{JOTA}

\begin{document}

\title{Asymptotic Linear Convergence of ADMM for Isotropic TV Norm Compressed Sensing\thanks{Communicated by Shoham Sabach.}}
\titlerunning{Local Convergence Rate for TV Compressed Sensing}

\author{Emmanuel Gil Torres \and Matt Jacobs \and Xiangxiong Zhang}

\institute{Emmanuel Gil Torres, Purdue University, 
              egiltorr@purdue.edu \\
              Matt Jacobs,   
              University of California, Santa Barbara,  
              majaco@ucsb.edu \\
              Xiangxiong Zhang,  
             Purdue University, 
              zhan1966@purdue.edu
}

\vspace{-3cm}

 \date{Received: date / Accepted: date}

\vspace{-1cm}

\maketitle

\begin{abstract}
 We prove an explicit local linear rate for ADMM solving the isotropic Total Variation (TV) norm compressed sensing problem in multiple dimensions, by 
    analyzing  the auxiliary variable in the equivalent  Douglas-Rachford splitting on a dual problem.  Numerical verification  on large 3D problems and real MRI data will be shown. {The proven rate is not sharp, but it provides an explicit upper bound that appears close to the observed convergence rate in numerical experiments, although we do not claim this behavior holds in general.}
\end{abstract}

\vspace{-1cm}

\section{Introduction}
\subsection{The Isotropic TV Norm Compressed Sensing}

The isotropic total variation (TV) norm compressed sensing (CS)  \cite{tvcs2015} is 
\begin{subequations}
\begin{align}
\label{eq:tvcs_primal-1}
    \min_{u} \|u\|_{TV}\quad \mbox{subject to} \quad  \hat u(\ell)=b_\ell,\quad \forall \ell\in \Omega=\{1, i_2, \cdots, i_m\}, 
\end{align}
where $u$ is a $d$-dimensional image of size $n_1\times n_2\times \cdots \times n_d=N$, $\hat u$ denotes the $d$-dimensional discrete Fourier transform of $u$, $\Omega$ is a set of observed frequency indices  with $m<N$, and 
  $b\in \mathbb{C}^m$ denotes the observed data. 
In \eqref{eq:tvcs_primal-1},  $1\in \Omega$ means that the observed data should include the zeroth frequency of $u$.

We   also regard $u$ as a vector $u\in \R^N\backsimeq \R^{n_1\times n_2\times \cdots \times n_d}$.
Let $\mathcal K  : \R^N \to [\R^N]^d$ denote the discrete gradient operator, which will be defined in Section \ref{sec-notation}.
Then
 the isotropic TV norm is defined as  $\|u\|_{TV} \coloneqq \|\mathcal K u\|_{1,2}$ and  $\|\cdot \|_{1,2}$ norm is
\begin{align}
\label{1-2-norm}
    \|v\|_{1,2} =  \sum_{j=1}^N \sqrt{\sum_{i=1}^d|v^i_j|^2},\quad   v= \begin{bmatrix}
        v^1 \\ \vdots \\ v^d
    \end{bmatrix}\in [\R^N]^d,\quad  v^i=\begin{bmatrix}
v^i_1 \\ \vdots \\ v^i_N
    \end{bmatrix}\in \R^N. 
\end{align}
 \end{subequations}
Note that this reduces to the classical $\ell^1$ norm for $\R^N$ 
when $d=1$.

For processing images,
the isotropic TV norm was introduced for  denoising in \cite{rudin1992nonlinear}, and used in many applications such as deconvolution and zooming, image in-painting and motion estimation \cite{chambolle2011first}, as well as compressed sensing \cite{candes2006robust}. Total Variation Compressed Sensing (TVCS) has been used practically in the areas of nuclear medicine and limited view angle tomosynthesis studies \cite{persson2001total,li2002accurate,kolehmainen2003statistical,velikina2007limited}.
Though in this paper we only focus on the Fourier measurements, e.g., MRI \cite{lustig2007sparse},
 the algorithm and our analysis may also be useful for applications using the Radon Transform \cite{leary2013compressed} and radio interferometry  \cite{wiaux2009compressed} since the sampling process can be modeled as samples of the Fourier transform \cite{leary2013compressed,wiaux2009compressed}.

\subsection{ADMM for TV Norm Minimization}

For  solving \eqref{eq:tvcs_primal-1},
we focus on the alternating direction method of multipliers (ADMM) \cite{fortin2000augmented}, and study  its asymptotic linear convergence rate. Though the local linear convergence has been established for ADMM solving TV norm minimization \cite{liang2017local} and \cite{aspelmeier2016local}, no explicit rates were given for the multi-dimensional case due to the fact that $\|\cdot\|_{1,2}$ is no longer locally polyhedral for $d\geq 2$. 
There are  other  popular first order splitting methods, such as the primal dual hybrid gradient (PDHG) method \cite{chambolle2011first}. For problem \eqref{eq:tvcs_primal-1}, it has been well known \cite{gabay1979methodes,gabay1983applications,glowinski1989augmented,esser2009applications,esser2010general} that ADMM is also equivalent to quite a few popular first order methods   with special choice of parameters, including Douglas-Rachford splitting (DRS) \cite{lions1979splitting} and split Bregman method \cite{goldstein2009split}. In Section \ref{sec:gprox-pdhg}, we will show that ADMM is also equivalent to G-prox PDHG method introduced in \cite{jacobs2019solving}, which was proven and  shown to be efficient for very large images.

ADMM can be applied to any problem of the form
    \begin{align}
    \label{eq:primala}
    \min_{u\in X} f(\mathcal K u) + g(u),
\end{align}
where $X,Y$ are two finite-dimensional real Hilbert spaces, $\mathcal K:X \to Y$ is a continuous linear operator,  and $g:X \to \R$ and $f: Y \to \R$ are proper, convex, and lower semi-continuous functions. 
To write problem \eqref{eq:tvcs_primal-1} in the form of problem \eqref{eq:primala}, we use
\begin{align*}
  X=\mathbb R^N,\quad   \mathcal K:\R^N \to [\R^{N}]^d, \quad f(v)=\|v\|_{1,2}, \quad g(u)=\iota_{\{u: Au=b\}}(u),
\end{align*}
where $\iota_C(u)=\begin{cases}
    0, & u\in C\\ +\infty, & u\notin C
\end{cases}$ is the indicator function of a set $C$, and $Au=b$ denotes measurements $\hat u_{\ell}=b_{\ell}, \ell\in \Omega$ in  \eqref{eq:tvcs_primal-1}.
  ADMM for \eqref{eq:primala} is described in Algorithm \ref{alg-ADMM}.
\begin{algorithm}[htbp]
\caption{ADMM with step size $\gamma$.}
\label{alg-ADMM}
\begin{algorithmic}[1]
    \STATE $x_{k+1} = \argmin_x \ g(x) + \langle \mathcal{K}x,z_k \rangle + \frac{\gamma}{2}\|\mathcal Kx-y_k\|^2 $
    \vspace*{0.1cm}
    \STATE $y_{k+1} = \argmin_y f(y) - \langle y,z_k\rangle + \frac{\gamma}{2} \|y -\mathcal Kx_{k+1}\|^2 $
    \vspace*{0.1cm}
    \STATE $z_{k+1} = z_k - \gamma \big(y_{k+1}-\mathcal Kx_{k+1} \big)$
\end{algorithmic}
\end{algorithm}

\subsection{The Main Result: A Local Linear Rate of ADMM}

The Fenchel dual problem of \eqref{eq:primala} can  be written  as
\begin{align}\label{eq:dual}
\min_{p\in \R^{N\times d}} f^*(p) + h^*(-p),\quad  h^*(-p):=g^*(-\mathcal K^*p),
\end{align}
where $f^*,g^*$ are convex conjugates of $f,g$, and $\mathcal K^*$ is the adjoint operator of $\mathcal K$.
For analyzing   Algorithm \ref{alg-ADMM}, we will consider the Fenchel dual problem to \eqref{eq:dual}. As shown in Appendix \ref{appendix-dual},  the dual problem of \eqref{eq:dual} can be given as
    \begin{equation}
    \label{eq:doubleDual}
        \min_{v\in \R^{N\times d}} f(v)+h(v),\quad f(v)=\|v\|_{1,2}, \quad h(v) =\iota_{\mathcal K \{u: Au=b\}}(v),
    \end{equation}     
  where $\mathcal K \{u: Au=b\}:=\{v: v=\mathcal Ku, Au=b\}$.  
It is well known that 
  the ADMM on \eqref{eq:primala} with a step size $\gamma$ is equivalent to   DRS on \eqref{eq:dual}  with a step size $\gamma$, which is also equivalent to 
  DRS on
  \eqref{eq:doubleDual} with a  step size $\frac{1}{\gamma}$ as reviewed in Appendix \ref{appendix-drs-primal-dual}.
  Next, we describe DRS solving \eqref{eq:doubleDual} which will be used to analyze Algorithm \ref{alg-ADMM}.
  Let $\iden$ be the identity operator.
Define the proximal and reflection operators with a step size $\tau>0$ respectively  as
\begin{align*}
    \prox_{f}^{\tau} (x) = \operatorname{argmin}_z  f(z)+\frac{1}{2\tau}\|z-x\|^2, \quad \refl_{f}^\tau=2\prox_{f}^\tau-\iden.
    \label{prox-def}
\end{align*}

DRS on problem \eqref{eq:doubleDual} is defined by a fixed point iteration of the operator $H_\tau=\frac{\iden + \refl_h^{\tau}\refl_f^{\tau}}{2}$. In particular,
 in Algorithm \ref{alg-DR-double-dual}, $q_k$ is an auxiliary variable and $v_k$ converges to the minimizer of \eqref{eq:doubleDual}.
 The equivalence between Algorithm \ref{alg-ADMM} and Algorithm \ref{alg-DR-double-dual} will be reviewed in Section \ref{sec:equivalence-implementation}.

The function  $f(v)=\|v\|_{1,2}$ is sparsity promoting \cite{santosa1986linear}, and its proximal operator $\prox_f^{\tau}$ is the well known Shrinkage operator in multiple dimensions. Let $S_\tau$ denote the shrinkage operator with step size $\tau$.
For any $q=[q^1 \ \cdots \ q^d]^T\in [\R^{N}]^d$ with $q^i=[q^i_1 \ \cdots \ q^i_N]^T\in \R^{N}$, we introduce the notation $q_j=[q_j^1 \ \cdots \ q_j^d]\in \R^d$ and we will call the subscript the {\it spatial index}. 
Then the shrinkage operator $\prox_f^{\tau}(q)=S_\tau(q) \in   [\R^{N}]^d$ can be expressed   as
\begin{equation}
 \label{shrinkage}
    \prox_f^{\tau}(q)_j = S_{\tau}(q)_j = \begin{cases}
          0, \quad &\mathrm{if} \ \|q_j\| \le \tau \\
        q_j - \tau \frac{q_j}{\|q_j\|}, &\mathrm{otherwise}
    \end{cases}.
\end{equation}

 \begin{algorithm}[htbp]
\caption{Douglas-Rachford splitting (DRS) on Problem \eqref{eq:doubleDual} with a step size $\tau>0$.}
\label{alg-DR-double-dual}
\begin{algorithmic}[1]
    \STATE  $q_{k+1} = H_\tau(q_k)=\frac{\iden+\refl_h^{\tau}\refl^{\tau}_f}{2}(q_k) = \prox_h^{\tau}(\refl^{\tau}_f(q_k)) + q_k - \prox_f^{\tau}(q_k)$
    \vspace*{0.1cm}
    \STATE $v_{k+1} = \prox_f^{\tau}(q_{k+1})$
\end{algorithmic}
\end{algorithm}

We need 
proper assumptions so that  \eqref{eq:doubleDual} 
has a unique minimizer.  
\begin{assumption}\label{asmp:sizeOmega}
Let $u_*$ be the true image, $s>0$ be a fixed  accuracy parameter, $\mathcal{K}u_*$ be the gradient of the image, and $\mathcal S$ be the support of $\mathcal{K}u_*$. Let $|\mathcal S|$ denote the number of nonzero entries in $\mathcal{K}u_*$. Assume  $\Omega$ is chosen uniformly at random from sets of size $|\Omega | = m \ge C_s^{-1}\cdot |\mathcal S| \cdot \log(N)$ for some constant $C_s$.
\end{assumption}

\begin{theorem}[Theorem 1.5 in \cite{candes2006robust}]
\label{thm:uniqueness}
    Under Assumption \ref{asmp:sizeOmega} in which $C_s \approx \frac{1}{23(s+1)}$ for $|\Omega| \le N/4, s\ge 2, \ \mathrm{and} \ N\ge 20$, with probability at least $1-O(N^{-s})$, the minimizer $v_*$ to  \eqref{eq:doubleDual} is unique and   $v_*=\mathcal{K}u_*$. 
\end{theorem}

Assume the minimizer $v_*$ to  \eqref{eq:doubleDual} vanishes at $r$ spatial indices, i.e., $(v_*)_j=[(v_*)_j^1 \ \cdots \ (v_*)_j^d]=0$ for $j=j_1, \cdots, j_r$. Let $e_i\in \R^N$   be the standard basis in $\R^N$. 
Denote the basis vectors corresponding to zero components
in $v_*$ as $e_i$ $(i=j_1, \cdots, j_r)$.
Let $B=[ e_{j_1}, \ ... \ , \ e_{j_r} ]^T \in \R^{r\times N}$ be the selector matrix of the zero components of $v_*$.
Let $\widetilde{B}$  be the block diagonal matrix
\begin{equation}\label{eq:bTildeMatrix}
    \widetilde{B} = \begin{pmatrix}
        B & & \\
         & \ddots & \\ && B

    \end{pmatrix}\in \R^{dr\times dN}.
\end{equation}

For Algorithm \ref{alg-DR-double-dual}, its fixed point $q_*$ is not unique, depending on the initial guess $q_0$, even if the minimizer $v_*$ to Problem \eqref{eq:doubleDual} is unique. 
Our main result is a local linear rate  of Algorithm \ref{alg-DR-double-dual} solving problem \eqref{eq:doubleDual} for {\it standard} fixed points  similar to the ones defined in \cite{demanet2016eventual}, in the sense of the following. 

\begin{definition}\label{def:qSpace}
{\it For TVCS \eqref{eq:tvcs_primal-1} with 
measurements denoted as $Au=b$,
consider its equivalent problem \eqref{eq:doubleDual}  with a solution $v_*$.
Let $\overline{\mathcal{B}_{\tau}(0)}$ be the closed ball in $\R^d$ of radius $\tau$ centered at $0$,
and $\overline{\mathcal{B}_{\tau}(0)}^c$ be its complement.
Define  
    \begin{align*}
        \mathcal{Q}_j &= \begin{cases}
        \overline{\mathcal{B}_{\tau}(0)}, \quad &\mathrm{if} \ (v_*)_j=\begin{bmatrix}
            (v_*)^1_j &
            \cdots & (v_*)^d_j 
        \end{bmatrix}^T= 0 \\
        \overline{\mathcal{B}_{\tau}(0)}^c, &\mathrm{if} \ (v_*)_j=\begin{bmatrix}
            (v_*)^1_j &
            \cdots & (v_*)^d_j 
        \end{bmatrix}^T\neq 0
    \end{cases}  \subset \R^d, \\
    \mathcal{Q} &= \{v\in   [\R^{N}]^d: S_{\tau}(v)_j =0 \iff (v_*)_j = 0  \} \backsimeq \mathcal{Q}_1 \oplus \hdots \oplus \mathcal{Q}_N,  
\end{align*} 
which is the preimage of the shrinkage operator \eqref{shrinkage} on  vectors with the same support set as $v_*$.
Let $q^0$ be the initial value in DRS, and $q_* = \lim_{k\to \infty} H_{\tau}^k(q^0)$.  
We call $(b, A; q_0)$ a standard problem for the DRS if $q_*$ belongs to
the interior of $\mathcal Q$. In this case, we call  $q_*$  an {interior fixed point}. 
Otherwise, we say that $(b, A; q_0)$
is nonstandard for DRS and that $q_*$
is a boundary fixed point.}
\end{definition}
 Now the main result of this paper can be stated as follows:
\begin{theorem}
\label{thm-main}
    Let $\theta_1$ be the smallest non-zero principal angle  between the two linear spaces $\mathcal{K}\ker{A}= \{ \mathcal{K}u: u \in \ker(A)\}$ and $\ker(\widetilde{B})$ with $\widetilde{B}$ defined in \eqref{eq:bTildeMatrix}. 
    Consider ADMM (Algorithm \ref{alg-ADMM}) solving \eqref{eq:tvcs_primal-1} with a step size $\gamma=\frac{1}{\tau}>0$, which is equivalent to DRS (Algorithm \ref{alg-DR-double-dual}) solving \eqref{eq:doubleDual} with a step size $\tau$. 
    The convexity of the problem \eqref{eq:doubleDual} implies that  DRS iterates $q_k$ converge to a fixed point $q_*$.
    Assume that   $q_*$ is  an interior fixed point.
    Under Assumption \ref{asmp:sizeOmega}, with probability $1-O(N^{-s})$,  for small enough $\tau>0$, there is an integer $K$ such that for all $k\ge K$, $ \|q_k - q_*\| \le \left[\cos\theta_1 + \max_{j:\|(v_*)_j\|\ne 0} \frac{2\tau}{\|(v_*)_j\|_2} \right]^{k-K} \|q_K - q_*\|.$
  
\end{theorem}
  We remark that 
the local linear rate above looks similar to the one proven for $\ell^1$-norm compressed sensing in \cite{demanet2016eventual}, but with two differences. The first difference is that the angle $\theta_1$ in this paper for the TV norm is  different from the angle in  \cite{demanet2016eventual}  due to the fact that the set $\mathcal Q$ is more complicated for TV norm.
The second difference is  the  term $\max_{j:\|(v_*)_j\|\ne 0} \frac{2\tau}{\|(v_*)_j\|} $, which arises only in multiple dimensions, $d\ge 2$. When $d=1$, this additional term can be removed in the proof and the main result proven in this paper reduces to  the same local linear convergence rate in \cite{demanet2016eventual}. 
{Hence, the novelty lies in providing an explicit upper bound for the local linear convergence rate of TVCS in higher dimensions. The bound holds for any TVCS problem satisfying Assumption \ref{asmp:sizeOmega}; however, it is not sharp for general   problems.} 
On the other hand, our provable rate is close to the observed local linear convergence rate in our numerical experiments that satisfy Assumption \ref{asmp:sizeOmega}, {even for 3D problems of size $512^3$, although we do not claim that this behavior holds for general large problems.}
This is illustrated in Figure \ref{fig:intro-linear}, which shows the local linear convergence of ADMM for two TVCS  problems: one is a 2D Shepp–Logan phantom of size $64\times64$ and the other one is a 3D phantom of size $512^3$. Although  our provable rate depends on the step size 
\(\gamma = \frac{1}{\tau}\),  the choice of step sizes does not seem to 
significantly affect the local linear convergence rate, e.g., the actual rate seems to be dominated by $\cos\theta_1$, as suggested by Figure~\ref{fig:intro-linear} and
also other numerical tests in Section~\ref{sec:tests_tau}.

\begin{figure}[h!]
    \centering
    \begin{subfigure}[b]{0.48\textwidth}
        \includegraphics[width=\textwidth]{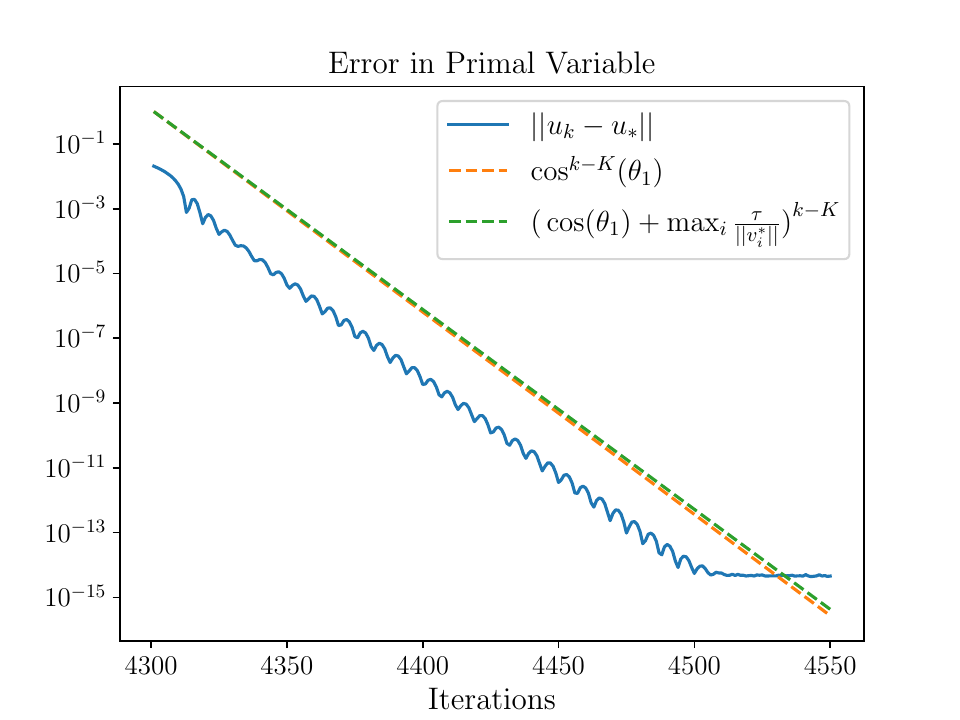}
    \end{subfigure}
    \hfill
    \begin{subfigure}[b]{0.48\textwidth}
        \includegraphics[width=\textwidth]{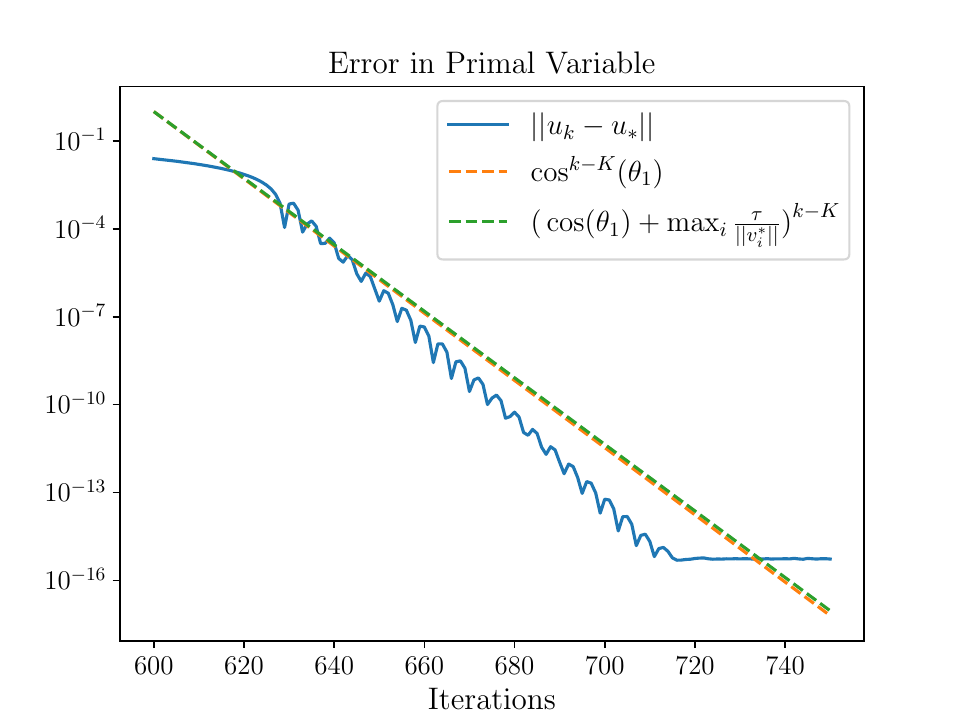}
    \end{subfigure}
    \caption{{The local linear rate of $u_k-u_*$ for TVCS. Here, $u_*$ is the true image, and $u_k$ is the image at $k$-th iteration of ADMM. Left: 2D Shepp–Logan phantom  $(64\times 64)$,  step size $\gamma = \frac{1}{\tau}=100$, $K=4300$. Right: 3D Shepp–Logan phantom  $(512^3)$, step size $\gamma=\frac{1}{\tau}=10$, $K = 600$. In both tests, about $30\%$ of the Fourier frequencies are observed.  }}
    \label{fig:intro-linear}
\end{figure}

\subsection{Related Work, Contributions and Outline}
Convergence rates of DRS and ADMM have been studied in different settings. In \cite{lions1979splitting}, a global linear convergence was shown when one of the two functions is strongly convex with a Lipschitz continuous gradient. In \cite{giselsson2016linear,davis2016convergence}, local linear convergence was shown under the assumptions of smoothness and strong convexity. For $\ell^1$-norm compressed sensing, local linear rate is related to the first principal angle between two subspaces in \cite{demanet2016eventual}. In \cite{boley2013local}, local linear convergence of ADMM was shown for quadratic and linear programs as long as the solution is unique and the strict complementary condition holds. {Additionally, local linear convergence of ADMM was also shown for quadratic and linear programs restricted to polyhedral sets in \cite{han2013local} using the affine variational inequality. In \cite{goldfarb2005second}, it was shown that the isotropic TV minimization can be reformulated as a Second-Order Cone Problem (SOCP). In particular, TVCS for $d=1$ can be posed as a linear program \cite{alizadeh2003second}, for which the results in \cite{han2013local} apply.}
By the idea of partial smoothness developed in \cite{lewis2002active}, the results of \cite{demanet2016eventual,bauschke2014rate,boley2013local} can be unified under a general framework in \cite{liang2017local}, which shows the existence of local linear convergence for many problems, and provides explicit convergence rates if the  cost functions  are locally polyhedral.  In \cite{aspelmeier2016local}, it was proved that applying DR or ADMM to composite  problems consisting of a convex function and a convex function composed with an injective linear map yields local linear rates.

The main contribution of this paper is to provide an explicit rate for the local linear convergence of ADMM applied to isotropic TV norm compressed sensing problem. Our explicit rate, albeit not sharp mathematically, provides some insights into the behavior of ADMM for TV norm minimization. On the other hand, the proven rate matches well with the observed rate for ADMM with a large step size $\gamma$ for large 3D problems such as real 3D MRI data. Moreover, while our proof is largely based on the work in \cite{demanet2016eventual}, we introduce some novel ideas for the isotropic TV norm which might be also useful for other problems such as second order cone programs. 
Our main techniques include exploiting the specific structure of the DRS fixed points for specific problems, and using the equivalences of algorithms to study the local linear convergence through the equivalent problem \eqref{eq:doubleDual}.  
Other contributions consist of adding the recently developed algorithm G-prox PDHG \cite{jacobs2019solving}, to the already known equivalencies among ADMM, DRS, and Split-Bregman method, which will be summarized in Table \ref{tab:relations} in Section \ref{sec:gprox-pdhg} with derivations in the Appendix \ref{appendix-GroxPDHG}.

The rest of this  paper is organized as follows.   Section \ref{sec-notation} contains some preliminaries and notation needed. In Section \ref{sec:equivalence-implementation}, we provide the equivalence between ADMM and G-prox PDHG for general problems and give an explicit implementation formula for the problem \eqref{eq:tvcs_primal-1}. In Section \ref{sec:main}, we provide the theorem and proof of our main result. Section  \ref{sec:tests} includes numerical experiments, which validate the theoretical results and show what performance we can expect for 2D and 3D problems. 
Section \ref{sec:remark} gives concluding remarks.

\section{Preliminaries}
\label{sec-notation}

\subsection{Notation and Preliminaries}

Let $\iden$ be the identity operator. Let $I$ be the identity matrix and $I_n$ denote the identity matrix of size $n\times n$. 
For any matrix $A$, $A^T$ denotes its transpose,   $A^*$ denotes its conjugate transpose and $A^+$ denotes its pseudo inverse. For a linear operator $\mathcal K$, $\mathcal K^*$ denotes its adjoint operator. For any $v =[v_1\ \cdots \ v_N]^T \in[\R^{N}]^d$, the $\|\cdot \|_{1,2}$ norm is defined in \eqref{1-2-norm}
and its dual norm is   $\|v\|_{\infty,2} = \max_{j=1, \hdots, N} \sqrt{ \sum_{i=1}^d |v^i_j|^2}. $ 
For convenience, we will also regard any $q \in [\R^{N}]^d$ as a vector in $\R^{Nd}$, then $\|q\|$ denotes the $2$-norm in $\R^{Nd}$.

All functions considered in this paper are {\it  closed,  convex,} and {\it proper} \cite{rockafellar,beck2017first}. A closed extended function is also a lower semi-continuous function \cite[Theorem 2.6]{beck2017first}.  
If  $C$ is a closed convex set, 
 the  indicator function $\iota_{C}(x)$ is a {closed convex proper} function thus also 
lower semi-continuous. For a function $f$, its {\it subgradient} is a set $\partial f(x)$.
We summarize a few useful results, see \cite{beck2017first}.
\begin{theorem}\label{eqn:pd_relation}
A closed convex proper function $f$ satisfies:
\begin{itemize}
    \item [(i)] $\prox^{\tau}_{f\circ (-\iden)}(x) = -\prox^{\tau}_f(-x).$
    \item [(ii)] $f^{**}(x) = f(x)$.
    \item [(iii)] $\langle x, y\rangle=f(x)+f^*(y) \Leftrightarrow x\in \partial f^*(y)\Leftrightarrow y\in \partial f(x)$.
   \item [(iv)] $x^* = \argmin_x \langle x, y^* \rangle + f(x) \ \iff \ -y^* \in \pd f(x^*).$
\item [(v)] {\it Moreau Decomposition}: $\prox_f^{\gamma}(x) + \gamma \prox_{f^*}^{\frac{1}{\gamma}}\Big(\frac{x}{\gamma}\Big) = x.$
\end{itemize}
\end{theorem} 
 
\vspace{-0.5cm}

\subsection{Discrete Fourier Transform and Differential Operators}
\label{sec-notation-3D}

{Let $\mathcal{F}$ denote the normalized discrete Fourier transform (DFT) matrix, and $\hat u = \mathcal{F}u \in \mathbb{C}^N$ denote the normalized discrete Fourier transform of $u\in \R^N\backsimeq \R^{n_1\times n_2\times \cdots \times n_d}$. Let $\check v$ denote the inverse DFT of $v$, then 
 $\check v=\mathcal F^* v.$ We have $\langle u, v\rangle_{\R^N}=\langle \mathcal F u, \mathcal F v\rangle_{\mathbb{C}^N},  \forall u, v\in \R^N$,  and  $\langle u, v\rangle_{\mathbb C^N}=\langle \mathcal F^* u, \mathcal F^* v\rangle_{\mathbb R^N}, \forall u, v\in \mathbb C^N$ satisfying $\mathcal F^* u, \mathcal F^* v
\in \mathbb R^N$. }

\paragraph{Notation for One-Dimensional Problems:}

{Let $T$ be the normalized 1D DFT matrix. Then, when $d=1$, $\widehat{u} = \mathcal{F}u = Tu$ and $u\in\R^N = \R^{n_1}$. For the discrete gradient operator, we consider the 1D periodic case. For $u \in \R^{n_1}$, we define the forward difference matrix as,
\begin{equation}
    K=\begin{pmatrix}
              -1 & 1 &  & & \\
                &\ddots & \ddots & & \\
                 & & -1 &1 \\
              1  & &   & -1
             \end{pmatrix}. \label{matrix-K}
\end{equation}
Then its transpose  $K^T$ approximates the negative derivative and $D=K^TK$ is the negative discrete Laplacian. For a one-dimensional image $u$, the operators $\mathcal K$ and $\mathcal K^*$ can be expressed as $\mathcal K u=Ku$ and $\mathcal K^* u=K^Tu$. 
Notice that the matrix $K$ in \eqref{matrix-K} is circulant. Hence, it can be diagonalized by DFT matrix, i.e., $K=T^*\Lambda T$ where $\Lambda$ is diagonal. This is one of the key properties that allows for an explicit implementation formula in Section \ref{sec:equivalence-implementation}.  }

\paragraph{Notation for Multi-Dimensional Problems:}
For multiple dimensions, we focus on $d=2$ as an example of introducing notation.
For simplicity, we assume $n_1=n_2$ 
for a two-dimensional image. For $U\in \R^{n\times n}$, let $u=\mathrm{vec}(U)\in \R^N$ be the column-wise vectorization of the matrix, then $(A\otimes B) u=\mathrm{vec}(BUA^T), \forall A, B\in \mathbb{C}^{n\times n}.$ {Now the DFT of a 2D image, $u = \mathrm{vec}(U) \in \R^N,$ is $\widehat{u} = \mathcal{F}u = (T\otimes T)u = \mathrm{vec}(TUT^T)$.}
Define the
discrete gradient and negative discrete divergence as follows,
\begin{align*}
    & \nabla_h u = \begin{pmatrix}
        K \otimes I \\ I \otimes K   
    \end{pmatrix} u= \begin{pmatrix}
        \mathrm{vec}(U K^T) \\ \mathrm{vec}(KU)   
    \end{pmatrix},  -\nabla_h\cdot \begin{pmatrix}
       u\\ v
    \end{pmatrix}   = \begin{pmatrix}
        K^T \otimes I &  I \otimes K^T
    \end{pmatrix}   \begin{pmatrix}
       u \\  v
    \end{pmatrix},
\end{align*}
 where $U, V\in \R^{n\times n}, u=\mathrm{vec}(U), v=\mathrm{vec}(V).$
The operators $\mathcal K$ and $\mathcal K^*$ can be expressed by $\mathcal K u=\nabla_h u \in \R^N\oplus \R^N,\ \forall u\in \R^N, $ and 
$\mathcal K^* p=-\nabla_h\cdot p\in \R^N,\ \forall  p\in \R^N\oplus \R^N.$
With the fact $$K\otimes I=(T^*\Lambda T)\otimes (T^* I T)=(T^*\otimes T^*)(\Lambda\otimes I)(T\otimes T)=
\mathcal F^* (\Lambda\otimes I) \mathcal F, $$
the operator $\mathcal K:\R^N \to \R^{2N} \cong \R^N \oplus \R^N$ can be decomposed as
\begin{align}
  \label{discreteGrad-Fourier}   \mathcal K &=\nabla_h= \begin{pmatrix}
        \mathcal F^* & 0 \\
        0 & \mathcal F^*
    \end{pmatrix}
 \begin{pmatrix}
        \Lambda \otimes I \\ I \otimes \Lambda   
    \end{pmatrix} \mathcal F= \begin{pmatrix}
        \mathcal F^* & 0 \\
        0 & \mathcal F^*
    \end{pmatrix} \BLambda \mathcal F,\quad \BLambda=\begin{pmatrix}
        \Lambda \otimes I \\ I \otimes \Lambda   
    \end{pmatrix},\\
    \label{discretediv-Fourier}    \mathcal  K^* &=-\nabla_h\cdot = \mathcal F^*\begin{pmatrix}
        \Lambda^* \otimes I & I \otimes \Lambda^*
    \end{pmatrix}
     \begin{pmatrix}
        \mathcal F & 0 \\
        0   & \mathcal F 
    \end{pmatrix}=\mathcal F^*\BLambda^*
     \widetilde{\mathcal F},\quad \widetilde{\mathcal F}=\begin{pmatrix}
        \mathcal F & 0 \\
        0   & \mathcal F 
    \end{pmatrix}.
\end{align}

The $d$-dimensional case can be defined similarly. We refer to \cite[Section 2.4]{liuGPU} for how to define $\mathrm{vec}(U)$ for a 3D image $U$. Let $K_{n}$ be the matrix in \eqref{matrix-K} of size $n\times n$, 
then consider the matrix constructed by one $K$ matrix and $d-1$ identity matrices via the Kronecker  product
\begin{equation}\label{eq:K_multid}
\mathcal K=\nabla_h=\begin{pmatrix}
    \mathcal K^1\\ \vdots \\ \mathcal K^d
\end{pmatrix},\quad \mathcal K^i=I_{n_1}\otimes \cdots \otimes K_{n_i}\otimes \cdots \otimes I_{n_d} \in \R^{N\times N}.
\end{equation}
Recall $K$ in \eqref{matrix-K} has an eigenvalue decomposition $K=T^*\Lambda T$. Let $\Lambda_{n}$ be the same diagonal eigenvalue matrix of size $n\times n$.
We construct the matrix
$$ \BLambda = \begin{pmatrix}
    \BLambda^1 \\ \vdots \\ \BLambda^d
\end{pmatrix}, \quad \BLambda^i=I_{n_1}\otimes \cdots \otimes\Lambda_{n_i}\otimes \cdots \otimes I_{n_d} \in \R^{N\times N},$$
and let $\lambda^i_{\ell}$ $(\ell=1,\cdots N)$ be the diagonal entries of $\BLambda^i.$

\subsection{The Constraint of Partially Observed Fourier Frequencies}\label{sec-notation-constraints}

For simplicity, we focus on the case $d=2$ and the discussion for $d\ge 3$ is similar.  
In \eqref{eq:tvcs_primal-1}, the constraint $ \hat u(\ell)=b_{\ell},\quad \forall \ell\in \Omega=\{1, i_2, \cdots, i_m\}$ can be denoted as
an  affine constraint $Au = b$ by a linear operator $A:\R^N \to \mathbb{C}^m$ with $m < N$, where the linear operator $A=M \mathcal F$ is a composition of a mask $M$ and the 2D DFT matrix $\mathcal F$ such that $\mathcal F \mathcal F^* = I$.
The mask matrix $M\in \mathbb R^{m\times N}$ is a submatrix of $I_{N}$.
We define $\Omega = \{1 ,i_2 \hdots, i_{m} \} \subset \{1,\hdots , N \}$ to be the indicator of which frequencies we know a priori, then $M = [ e_{1}; \ e_{i_2}; \ ... \ ; \ e_{i_{m}} ]^T\in \R^{m\times N}$, where $e_{\ell}$ are the standard basis vectors in $\R^N$.
 Notice, $AA^* = I_{m\times m}$, hence its pseudo inverse is $A^+ = A^*$. For convenience, we will use the notation
\begin{equation}
    \label{M-tilde}
     \widetilde{M}=\begin{pmatrix}
    M & 0 \\
    0 & M
\end{pmatrix},\quad \widetilde{A} = \begin{pmatrix}
    A & 0 \\
    0 & A
\end{pmatrix}=\begin{pmatrix}
    M\mathcal F  & 0 \\
    0 & M\mathcal F
\end{pmatrix}.
\end{equation}
Since $M$ is a submatrix of $I_{N}$, $M^*=M^T$.
Since $\Lambda \otimes I\in \R^{N\times N}$ is a diagonal matrix,  $M^* M (\Lambda \otimes I)$ is a diagonal matrix of size $N\times N$.
Therefore, we have $M^* M (\Lambda \otimes I)=[M^* M (\Lambda \otimes I)]^T= (\Lambda \otimes I)M^* M$, thus
\begin{equation}
    \label{m-matrix_property}
    \widetilde{M^*}\widetilde{M}\BLambda=\begin{pmatrix}
    M^* & 0 \\
    0 & M^*
\end{pmatrix}\begin{pmatrix}
    M & 0 \\
    0 & M
\end{pmatrix}\begin{pmatrix}
        \Lambda \otimes I \\ I \otimes \Lambda   
    \end{pmatrix}=\BLambda M^*M.
\end{equation}
Similarly,   $\BLambda^* \BLambda = \Lambda^*\Lambda \otimes I + I \otimes \Lambda^*\Lambda$ is  a diagonal matrix, thus 
\begin{equation}
    \label{m-matrix_property2} M^*M(\BLambda^*\BLambda)^{+} = (\BLambda^*\BLambda)^{+} M^*M.\end{equation}

\section{Equivalence to G-prox PDHG and an   Implementation Formula}
\label{sec:equivalence-implementation}

 \subsection{The Equivalence Between ADMM and G-prox PDHG}
 \label{sec:gprox-pdhg}

In this section, we mention when ADMM and G-prox PDHG are equivalent, mention their strengths and weaknesses, give an equivalent primal dual formulation of ADMM, and then provide an implementation formula for the TVCS problem. {It is worth noting that in general ADMM solves the more general problem $$\min_{u,v} f(v) + g(u) \quad \mbox{subject to} \quad  \mathcal{K}u + \mathcal{L}v =a.$$ Thus, ADMM has the strength of being applicable to a wider class of problems and being able to use the theory of variational inequalities or affine variational inequalities, as done in \cite{han2013local}. If $\mathcal{L} = - \iden$ and $a=0$, then we recover problem \eqref{eq:primala} for which we can show that ADMM and G-prox PDHG are equivalent. This is stated in Theorem \ref{thm-equivalence-GproxPDHG}, which will be proven in 
 Appendix \ref{appendix-GroxPDHG}. On problem \eqref{eq:primala}, G-prox PDHG has the advantage of the analysis in \cite{jacobs2019solving} where ergodic convergence of the cost function is established in the setting of general Hilbert spaces, possibly infinite-dimensional.} {While the equivalence on problem \eqref{eq:primala} provides some new theoretical insights 
 from provable results of G-prox PDHG in \cite{jacobs2019solving} to understanding ADMM, it does not yield essential practical advantages for implementing G-prox PDHG rather than ADMM, or conversely.} {If comparing ADMM in Algorithm \ref{alg-ADMM} to G-prox PDHG in Algorithm \ref{alg-Gprox-PDHG} in terms of implementation, we can see that their first steps are the same, and their second steps involve the proximal operator of $f$ and $f^*$ respectively, which are equivalent via Moreau's decomposition.  Thus there are neither advantages nor disadvantages if comparing Gprox PDHG to ADMM in terms of implementation. Both G-prox PDHG and ADMM directly track not only the primal but also dual variables, while DRS is naturally expressed in terms of a single auxiliary variable.}  There are many known equivalent yet seemingly different formulations of the ADMM in Algorithm \ref{alg-ADMM}.
We provide a summary of the variables that are equivalent in these algorithms in Table \ref{tab:relations}. These relations can be modified to extend to the generalized forms of these algorithms.

\begin{algorithm}[htbp]
\caption{G-prox PDHG with step sizes $\tau, \sigma>0$. 
\newline Initial guess $u_0\in \R^N, v_0, w_0\in [\R^{N}]^d$.}
\label{alg-Gprox-PDHG}
\begin{algorithmic}[1]
    \STATE $u_{k+1} = \argmin_u g(u) + \langle \mathcal Ku,w_k \rangle + \frac{1}{2\tau}\|\mathcal K(u-u_k)\|^2  $
    \vspace*{0.1cm}
    \STATE $ v_{k+1} = \argmax_v -f^*(v) + \langle \mathcal Ku_{k+1},v\rangle - \frac{1}{2\sigma} \|v -v_k\|^2$
    \vspace*{0.1cm}
    \STATE $w_{k+1} = 2 v_{k+1} -v_{k} $
\end{algorithmic}
\end{algorithm}

\begin{theorem}
\label{thm-equivalence-GproxPDHG}
   Algorithm \ref{alg-ADMM} (ADMM) with a step size $\gamma>0$ is equivalent to Algorithm \ref{alg-Gprox-PDHG} (G-prox PDHG) with $\tau=\frac{1}{\sigma}=\frac{1}{\gamma}$ via the change of variables:
  $ u_k: = x_k, \quad p_k: = z_k, \quad
	 w_k: =\frac{1}{\tau} \mathcal Kx_k + z_k - \frac{1}{\tau} y_k.$
\end{theorem}

\subsection{An Explicit Implementation Formula of G-prox PDHG}

For any vector $v= \begin{bmatrix}
        v^1  \hdots  v^d
    \end{bmatrix}^T \in [\R^N]^d$ with $v^i=\begin{bmatrix}
v^i_1  \hdots  v^i_N
    \end{bmatrix}^T \in \R^N$, let $v_j$ denote $v_j=\begin{bmatrix}
v^1_j & \cdots & v^d_j
    \end{bmatrix}^T \in \R^d$. Define $|v|:=\begin{bmatrix}
    \|v_1\|  \hdots  \|v_N\|
\end{bmatrix}^T \in \R^N$, and $\frac{v^i}{\max(1,|v|)}=\begin{bmatrix}
v^i_1/\max(1,\|v_1\|)  \hdots  v^i_N/\max(1,\|v_N\|)
    \end{bmatrix}^T\in \R^N.$ Then we define, $$\frac{v}{\max(1, |v|)}:=\begin{bmatrix}
    v^1/\max(1,|v|) \\ \vdots \\ v^d/\max(1,|v|)
\end{bmatrix}\in [\R^N]^d.$$
Let $\overline{\lambda}$ denote the complex conjugate of $\lambda$. For $w_k\in [\R^N]^d$, where $k$ will be the iteration index, we also denote it by $w_k=\begin{bmatrix}
    (w_k)^1 & \cdots & (w_k)^d
\end{bmatrix}^T$ with $(w_k)^i\in \R^N$
and $\widehat{(w_k)^i}$ being the $d$-dimensional discrete Fourier transform of $(w_k)^i$.
With the notation in Section \ref{sec-notation-3D}, 
for the TV compressed-sensing problem \eqref{eq:tvcs_primal-1}, Algorithm \ref{alg-Gprox-PDHG} can be explicitly implemented in Fourier space as described by Algorithm \ref{alg-Gprox-PDHG-2}. The derivation of Algorithm \ref{alg-Gprox-PDHG-2} will be given in Appendix \ref{appendix-implementation}. {For general convex functions \(f\), \(g\), and a linear operator \(\mathcal{K}\), the implementation of G-prox PDHG (Algorithm~\ref{alg-Gprox-PDHG}), which is equivalent to ADMM, can be difficult in practice. For the TVCS problem, its structure allows several simplifications. The first step of Algorithm~\ref{alg-Gprox-PDHG}   becomes explicit due to the facts that each block \(\mathcal{K}^i\) in \eqref{eq:K_multid} is diagonalizable by the \(d\)-dimensional DFT, and the constraint \(Au=b\) can be enforced in the Fourier domain. 
 Second, since \(f(v)=\|v\|_{1,2}\) admits a closed-form proximal operator, the second step of Algorithm~\ref{alg-Gprox-PDHG} has an explicit update.
Together, these properties make G-prox PDHG (or equivalently ADMM) explicitly implementable for TVCS. 
}

\begin{algorithm}[htbp]
\caption{An implementation formula of G-prox PDHG with a step size $\tau>0$ and $\sigma=\frac{1}{\tau}$ (or equivalently ADMM in Algorithm \ref{alg-ADMM} with $\gamma=\frac{1}{\tau}$) for TV norm compressed sensing.
\newline Initial guess: $u_0\in \R^N, v_0, w_0\in [\R^{N}]^d$
}
\label{alg-Gprox-PDHG-2}
\begin{algorithmic}[1]
    \STATE 
$\begin{cases}
       \widehat{u_{k+1}}(\ell) &= b_{\ell}, \  \ \ell \in  \Omega \\
\widehat{u_{k+1}}(\ell) &=        \widehat{u_k}(\ell) - \tau \left[ \sum\limits_{i=1}^d \overline{\lambda^i_{\ell}}\widehat{(w_k)^i}(\ell)\right]/\left[\sum\limits_{i=1}^d|\lambda_{\ell}^i|^2\right], \  \ell \notin \Omega
\end{cases} $
    \STATE $v_{k+1} =  \frac{v_{k} + \sigma \mathcal K u_{k+1}}{\max (1, |v_{k} + \sigma \mathcal K u_{k+1}|)},$
    \STATE $w_{k+1} = 2 v_{k+1} -v_{k}$.
    \vspace{0.1cm}
\newline Notation:  $k$ is the iteration index and $\ell$ is the frequency index. For any $w_k\in [\R^{N}]^d$, let
$w_k=\begin{bmatrix}
    (w_k)^1 & \cdots & (w_k)^d
\end{bmatrix}^T$ with $(w_k)^i\in \R^N$,  then $\widehat{(w_k)^i}$ denotes the discrete Fourier transform of $(w_k)^i$, and $\widehat{(w_k)^i}(\ell)$ denotes the component of $\widehat{(w_k)^i}$ at the $\ell$-th frequency.
 As defined in Section \ref{sec-notation-3D}, $\lambda_{\ell}^i$ $(\ell=1,\cdots,N)$ are diagonal entries of  $\BLambda^i$.
\end{algorithmic}
\end{algorithm} 

\begin{table}[htbp!]
    \caption{A summary of equivalent variables in ADMM, DRS and G-prox PDHG algorithms with proper step sizes: variables in each row are equivalent.}
    \label{tab:relations}
    \vspace*{0.05 cm}
    \centering
    \begin{tabular}{|l|c|c|c|}
    \hline
Method         &  ADMM for  \eqref{eq:primala}  &  
Douglas-Rachford for \eqref{eq:doubleDual}  &   G-prox PDHG for \eqref{eq:primala} \\  \hline 
Formula         &  Alg. \ref{alg-ADMM}   for \eqref{eq:primala} &  
Alg. \ref{alg-DR-double-dual}   for \eqref{eq:doubleDual} &   Alg. \ref{alg-Gprox-PDHG} for \eqref{eq:primala} \\  \hline 
Step Size &$\gamma={\frac{1}{\tau}}$ & $\tau$ & ${\sigma = \frac{1}{\tau}}$ \\
\hline        Primal Iterate  & $\mathcal{K}x_k$  & $q_k - (q_{k-1}-v_{k-1})$ & $\mathcal{K}u_{k}$ \\ 
         \hline
         Dual Iterate & $z_k$ & $\frac{q_k-v_k}{\tau}$ & $p_k$ \\ \hline
         Extragradient & $\frac{1}{\tau}\mathcal{K}x_k +  z_k - \frac{1}{\tau}y_k$  & $ \frac{2q_k - q_{k-1}}{\tau} - \frac{2v_k - v_{k-1}}{\tau}$ & $w_k$\\[0.2ex] \hline
    \end{tabular}
\end{table}

\section{The Main Result on an Explicit Local Linear Rate}
\label{sec:main}

We prove the main result in this section. For simplicity,  we focus on the case $d=2$,  and extensions to higher dimensions are straightforward. 
\vspace{-0.5cm}
\subsection{DRS on the Equivalent Problem}
In order to analyze the local linear convergence of ADMM, we will utilize some of the equivalent formulations.  
Recall that TVCS problem \eqref{eq:tvcs_primal-1} can be written as the primal formulation
\eqref{eq:primala}, and its Fenchel dual formulation is given as
\eqref{eq:dual}. The dual formulation of \eqref{eq:dual} can be written as \eqref{eq:doubleDual}.
We first make a few remarks about total duality. We have strong duality between the primal and dual problem due to Slater's conditions, which are satisfied if $\exists \ x$ s.t. $x\in\mathrm{ri} (\dom f) = \R^N  $ and  $Ax=b$. { Here, $\mathrm{ri}$ stands for the relative interior, and the relative interior of a set $C\subset \R^N$ is defined as $$\mathrm{ri}(C) = \{ x\in C: \mathcal{B}_r(x) \cap \mathrm{aff}(C) \subset C \ \mathrm{for} \ \mathrm{some} \ r>0 \}, $$ and $\mathrm{aff}(C)$ denotes the affine hull of the set $C$, i.e., the set of all affine combinations of points in $C$.  } For strong duality between \eqref{eq:dual} and \eqref{eq:doubleDual}, Slater's conditions
are satisfied by choosing $p=0\in \R^{2N}$ which implies $\|p\|_{\infty,2} < 1$, i.e., $p = 0 \in \mathrm{ri}(\dom f^*)$, and $\mathcal K^*p \in \Im(A^*)$.     To show total duality, we need the existence of a solution of \eqref{eq:doubleDual}. By Theorem \ref{thm:uniqueness}, 
    under Assumption \ref{asmp:sizeOmega}, with high probability, \eqref{eq:doubleDual} has a unique minimizer.
    Thus total duality holds.

\vspace{-0.3cm}
\subsection{The Proximal Operators} \label{sec:linearConv}
For the two functions $f$ and $h$ in \eqref{eq:doubleDual}, we need their proximal operators for studying DRS.
Since the function $h(v)=\iota_{\mathcal K\{u:Au = b  \}}(v)$ is an indicator function to an affine set, the proximal operator is the Euclidean projection to the set. With the derivation shown in Appendix \ref{appendix-dual}, the projection formula can be given as
\begin{align}
\label{prox-h}
    \prox_h^{\tau}(q) &= \widetilde{\mathcal{F}}^* (I-\Mt^* \Mt)\Sigma \widetilde{\mathcal{F}}q + \widetilde{\mathcal{F}}^*\Mt^* \Mt \BLambda M^*b.
\end{align}
Here, $\Sigma = \BLambda (\BLambda^* \BLambda)^+ \BLambda^*$ and $ (\BLambda^* \BLambda)^+$ denotes the pseudo inverse of $\BLambda^* \BLambda$. Next, we discuss $S_{\tau}$.

\begin{definition}\label{def:nonLinearOperator}
     For any $q = [ q^1 \cdots q^d]^T \in [\R^{N}]^d$ with $q^i=[ q^i_1 \cdots q^i_N]^T\in \R^N$, 
which can also be represented by $q_j=[ q^1_j \cdots q^d_j]\in \R^d$ with a spatial index $j=1,\cdots, N$, define an operator $\mathcal{N}: [\R^{N}]^d \to [\R^{N}]^d$ via the spatial index as
    \begin{equation*}
       \mathcal{N}(q)_j = \begin{cases} 0, \ & \mathrm{if} \  q_j = 0 ,\\ 
\frac{q_j}{\|q_j\|} \ & \mathrm{otherwise}\end{cases}\in \R^d,\quad j=1,\cdots, N.
    \end{equation*}
\end{definition}
Recall that we have defined $B=[ e_{j_1}, \ ... \ , \ e_{j_r} ]^T \in \R^{r\times N}$ to be the selector matrix of the zero components of $v_*$.
For any $q\in \mathcal{Q}$, with $\mathcal Q$ in Definition \ref{def:qSpace},
it is straightforward to verify that the shrinkage operator can be written as
\begin{equation}
\label{shrinkage-2}
S_{\tau}(q) = (I-\widetilde{B}^+\widetilde{B})(q - \tau \mathcal{N}(q)), \quad \forall q\in \mathcal Q,\end{equation}
in which we regard $q$ and  $\mathcal N(q)$ as column vectors in $\R^{Nd}$.
  
\begin{lemma}\label{thm:KNA_NB_intersection}
    Under Assumption \ref{asmp:sizeOmega}, with probability $1 - O(N^{-s})$, $\mathcal K\ker(A) \cap \ker(\widetilde{B}) = \{0\}$, where   $\mathcal{K}\ker (A) = \{v\in \R^{N\times d} : v = \mathcal Ku, \ \   u\in \ker (A) \}$.
\end{lemma}
\begin{proof}
   Consider any $z \in \mathcal{K}\ker(A) \cap \ker(\widetilde{B})$. First, $$z\in \mathcal{K}\ker(A)\Rightarrow z = \mathcal{K}u, u\in \ker(A).$$ By the fact that $MM^* = I_{m \times m }$ and the notation in \eqref{discreteGrad-Fourier} and \eqref{M-tilde},
\[ \widetilde{A}\mathcal K=\begin{pmatrix}
    M\mathcal F  & 0 \\
    0 & M\mathcal F
\end{pmatrix}  \begin{pmatrix}
        \mathcal F^* & 0\\
        0 & \mathcal F^*
    \end{pmatrix} \BLambda \mathcal{F}=\begin{pmatrix}
    M   & 0 \\
    0 & M 
\end{pmatrix} \BLambda \mathcal{F}=\widetilde{M}  \BLambda \mathcal{F}=(\widetilde{M} \widetilde{M^*})\widetilde{M} \BLambda \mathcal{F}. \]
By the property \eqref{m-matrix_property} and $u\in \ker(A)$, we have 
\begin{align*}
    \widetilde{A}z &= \widetilde{A}\mathcal Ku 
    = \widetilde{M} \widetilde{M^*}\widetilde{M}\BLambda \F u = \widetilde{M}\BLambda M^*M \F u = \widetilde{M}\BLambda M^*Au = 0.
\end{align*}
Second,
$z\in \ker(\widetilde{B})$ implies
 the support of $z$ is contained in the same support, $S$, as the unique solution $v_*$ to \eqref{eq:doubleDual}. Let $A_S$ denote the partial Fourier Transform restricted to signals with the support included in the set $S$.  Then
 $$\begin{pmatrix}
     A_S & 0 \\
     0 & A_S
 \end{pmatrix}z=\begin{pmatrix}
     A & 0 \\
     0 & A
 \end{pmatrix}z= \widetilde{A}z =  0.$$ By Theorem 3.1 in \cite{candes2006robust}, $A_S$ is injective, which implies $z = 0.$ 
\qed    
\end{proof}

\begin{remark}
For $\ell^1$-norm compressed sensing,  there are necessary \cite{zhang2015necessary} and  sufficient \cite{fuchs2004sparse} conditions to ensure a unique solution to \eqref{eq:tvcs_primal-1}, and the same techniques can be used to show $\mathcal{K} \ker(A)\cap \ker(\widetilde{B}) = \{0\}$ for one-dimensional TVCS problem, i.e.,  Problem \eqref{eq:tvcs_primal-1} with $d=1$. However, such a proof breaks down for $d\ge 2$. As shown in Lemma \ref{thm:KNA_NB_intersection} above, $\mathcal{K} \ker(A)\cap \ker(\widetilde{B}) = \{0\}$ can be ensured by the robust uncertainty principle in \cite{candes2006robust}.
\end{remark}

\subsection{Characterization of the Fixed Points to DRS}
 
For the function $h(v) =\iota_{\mathcal K\{u:Au = b  \}}(v)$, we have $$\partial h(q)  = \{w:\mathcal K^*w \in \Im(A^*)\} = (\mathcal K^*)^{-1}\big[\Im (A^*) \big],$$ 
where  $(\mathcal K^*)^{-1}\big[\Im (A^*)\big]$ denotes the pre-image of $\Im (A^*)$ under the operator $\mathcal K^*$.
By the optimality condition of \eqref{eq:doubleDual}, its minimizer $v_*$ satisfies $0\in \pd f(v_*) +(\mathcal K^*)^{-1}\big[\Im (A^*) \big]$, therefore $\pd f(v_*) \cap(\mathcal K^*)^{-1}\big[\Im (A^*) \big] \ne \emptyset$. 
Any vector $\eta \in \pd f(v_*) \cap(\mathcal K^*)^{-1}\big[\Im (A^*) \big]$ is called a {\it dual certificate}. The subgradient of $f = \| \cdot \|_{1,2}$ is given below as 
\begin{equation}\label{subderiv-f}
    \partial f(v_*) = \left\{ w \in \R^{Nd}: w_j \in \begin{cases}   \frac{(v_*)_j}{\|(v_*)_j\|} \quad & \mathrm{if} \ (v_*)_j \ne 0  \\ \mathcal{B}_1(0) \quad &\mathrm{else}\end{cases} \right\} .
\end{equation}

Theorem \ref{thm:uniqueness} (Theorem 1.5 in \cite{candes2006robust}) gives existence and uniqueness of the minimizer $v_*$, which implies the existence of a dual certificate.

\begin{lemma}\label{lem6}
    The set of fixed points of DRS iteration operator  $H_\tau=\frac{\iden+\refl_h^{\tau}\refl^{\tau}_f}{2}$ for the problem \eqref{eq:doubleDual} is given by:
    \begin{align*}
        \{v_* + \tau \eta : \eta \in \pd f(v_*) \cap(\mathcal K^*)^{-1}\big[\Im (A^*) \big]  \},
    \end{align*}
    and the fixed point is unique if and only if $(\mathcal K^*)^{-1}\big[\Im (A^*) \big]\cap \Im(\widetilde{B}^T) = \{ 0\}$ where $\widetilde{B}^T$ is the transpose matrix of  $\widetilde{B}$ with $\widetilde{B}$ defined in \eqref{eq:bTildeMatrix}.
\end{lemma}

\begin{proof}
Consider any $\eta \in \pd f(v_*) \cap(\mathcal K^*)^{-1}\big[\Im (A^*) \big] $. First, since $S_{\tau}$ is the proximal operator of $f(v)$, $\eta \in \pd f(v_*)$  implies   $S_{\tau}(v_* + \tau \eta) = v_*$. Second, by \eqref{discretediv-Fourier} and $A=M\mathcal F$,   we have 
$$\eta \in(\mathcal K^*)^{-1}\big[\Im (A^*) \big]\Rightarrow 
\mathcal K^*\eta\in \Im (A^*)
\Rightarrow  {\mathcal F}^*\BLambda^*  \widetilde{\mathcal F} \eta\in \Im (\mathcal F^* M^*)$$
\[\Rightarrow  \BLambda^*\widetilde{\mathcal F}\eta \in \Im (M^*)
\Rightarrow (I-M^*M)\BLambda^*  \widetilde{\mathcal{F}}\eta = 0.\]
By  \eqref{m-matrix_property} and \eqref{m-matrix_property2}, we have
$$(I-\Mt^* \Mt)\BLambda (\BLambda^* \BLambda)^+ \BLambda^*= \BLambda(\BLambda^* \BLambda)^+(I-M^*M)\BLambda^*$$
\[ \Rightarrow(I-\Mt^* \Mt)\BLambda (\BLambda^* \BLambda)^+ \BLambda^* \widetilde{\mathcal{F}}\eta =\BLambda(\BLambda^* \BLambda)^+(I-M^*M)\BLambda^*  \widetilde{\mathcal{F}}\eta = 0. \]
Since $v_*=\mathcal K u_*$   and $Au_*=b$, by \eqref{prox-h} and \eqref{discreteGrad-Fourier}, we have
\begin{align*}
     \prox_h^{\tau}(v_*-\tau \eta ) 
    = &\widetilde{\mathcal{F}}^* (I-\Mt^* \Mt)\BLambda (\BLambda^* \BLambda)^+ \BLambda^* \widetilde{\mathcal{F}}(v_*-\tau \eta) + \widetilde{\mathcal{F}}^*\Mt^* \Mt \BLambda M^*b \\
     =& \widetilde{\mathcal{F}}^* (I-\Mt^* \Mt)\BLambda (\BLambda^* \BLambda)^+ \BLambda^* \widetilde{\mathcal{F}}\mathcal K u_* + \widetilde{\mathcal{F}}^*\Mt^* \Mt \BLambda M^*A u_* \\
     = & \widetilde{\mathcal{F}}^* (I-\Mt^* \Mt)\BLambda(\mathcal F u_*) + \widetilde{\mathcal{F}}^* \Mt^* \Mt \BLambda \mathcal F u_* \\
     = & \widetilde{\mathcal{F}}^*\BLambda \mathcal F u_* = \mathcal Ku_* = v_*.
\end{align*}
Moreover, $\prox_f^{\tau}(v_* + \tau \eta)  = v_*$ implies   $\refl_f^\tau (v_* + \tau \eta)=v_* - \tau \eta$. Thus,
\begin{align*}
    H_{\tau}(v_* + \tau \eta) &= \prox_h^{\tau}(\refl_f^\tau (v_* + \tau \eta)) + v_* + \tau \eta - \prox_f^{\tau}(v_* + \tau \eta) \\
    &= \prox_h^{\tau}(v_* - \tau \eta) + \tau \eta= v_* + \tau \eta.
\end{align*}

Conversely, suppose $H_{\tau}(q) = q$ and define $\eta = \frac{q - v_*}{\tau}$. We want to show $\eta \in \pd f(v_*) \cap (\mathcal K^*)^{-1}\big[\Im (A^*)\big]$. 
By the convergence of the DRS iteration \cite{lions1979splitting},  $\prox_f^{\tau}(q) = v_*$, which implies that $\eta = \frac{q-v_*}{\tau} \in \pd f(v_*)$. Second, $H_{\tau}(q) = q$ and $\prox_f^{\tau}(q) = v_*$ implies $v_* = \prox_h^{\tau}(2v_* -q)$, which gives $-\eta \in \pd h(v_*) = (\mathcal K^*)^{-1}\big[\Im (A^*)\big]$ thus $\eta \in (\mathcal K^*)^{-1}\big[\Im (A^*)\big]$. 

To discuss uniqueness, let $q_1=v_*+\tau \eta_1, q_2=v_*+\tau \eta_2$ be two fixed points of $H_{\tau}$. Then $q_1 - q_2 = \tau (\eta_1 - \eta_2)$, where $\eta_1, \eta_2 \in\partial f(v_*)\cap  (\mathcal K^*)^{-1}\big[\Im (A^*)\big]$. From \eqref{subderiv-f}, note that $\eta_1,
\eta_2 \in \pd f(v_*)$  implies that $\pm (\eta_1- \eta_2) \in \Im(\widetilde{B}^T)$. Hence, $q_1 - q_2  \in (\mathcal K^*)^{-1}\big[\Im (A^*)\big] \cap \Im(\widetilde{B}^T)$, so  the fixed point is unique if and only if   $(\mathcal K^*)^{-1}\big[\Im (A^*)\big] \cap \Im(\widetilde{B}^T) = \{0 \}$. 
    \qed
\end{proof}

\subsection{Characterization of the DRS operator $H_{\tau}$}

Next, we estimate the nonlinear DRS operator $H_{\tau}$.

\begin{lemma}\label{lem7} For any fixed point $q_*$ of $H_\tau=\frac{\iden+ \refl^\tau_h\refl^\tau_f}{2}$, it satisfies 
    $$\|(I-\widetilde{B}^+\widetilde{B})\big(\mathcal{N}(q)-\mathcal{N}(q_*)\big)\| \le \max_{j:\|(v_*)_j\|\ne 0}\frac{2}{\|(v_*)_j\|} \|q-q_*\|,\quad \forall q\in [\R^{N}]^d \backsimeq \R^{Nd},$$
    where  $\|\cdot\|$ is the 2-norm in $\R^{Nd}$.
\end{lemma}
\begin{proof}
By Definition \ref{def:nonLinearOperator} and the definition of $\widetilde B$ in \eqref{eq:bTildeMatrix}, we have
\begin{align*}
    &\|(I-\widetilde{B}^+\widetilde{B})\big(\N (q)-\N(q_*)\big)\|^2  = 
    \sum_{i: (v_*)_i\ne 0} \left  \|\frac{q_i}{\|q_i\|} - \frac{(q_*)_i}{\|(q_*)_i\|} \right \|^2.
\end{align*}
For any nonzero $a, b\in \R^d$, we have 
$\left\|\frac{a}{\|a\|}-\frac{b}{\|b\|} \right\| \leq     \left\|\frac{a}{\|a\|}-\frac{b}{\|a\|} \right\|+  \left\|\frac{b}{\|a\|}-\frac{b}{\|b\|} \right\|   =  \frac{1}{\|a\|}  \left\|a-b \right\|+ \|b\|  \left| \frac{  \|b\|-\|a\| }{\|a\| \|b\|} \right| 
 =  \frac{1}{\|a\|} \Big(  \left\|a-b \right\|+   \big| \|b\|-\|a\|  \big|  \Big) \leq   \frac{2}{\|a\|}  \left\|a-b \right\|.$
By Lemma \ref{lem6}, $q_*=v_*+\tau \eta$ for a dual certificate $\eta$. 
For any index $i$ satisfying $(v_*)_i \ne 0$, we have
$(q_*)_i = (v_*)_i + \tau \frac{(v_*)_i}{\|(v_*)_i\|}$, which is implied by $\eta\in \partial \|v_*\|_{1,2}$. Hence, $\|(q_*)_i\| \ge \|(v_*)_i\|$. If we also use the inequality above with $a = q_i, \ b = (q_*)_i$, we obtain $ \left  \|\frac{q_i}{\|q_i\|} - \frac{(q_*)_i}{\|(q_*)_i\|} \right \|\le \frac{2}{\|(v_*)_i\|}\|q_i - (q_*)_i\|.$ 
   \qed
\end{proof}

\begin{lemma}\label{lem8}
    For any $q\in \mathcal{Q}$   
    and any DRS fixed point $q_*$,  
        $$H_{\tau}(q) - H_{\tau}(q_*) = \widetilde{H} (q - q_*) + \tau \Big[ (I-2C)(I-\widetilde{B}^+\widetilde{B}) \Big] (\mathcal{N}(q) - \mathcal{N}(q_*)),$$ where $\widetilde{H} = \Big[C(I-\widetilde{B}^+\widetilde{B}) +(I-C)\widetilde{B}^+\widetilde{B} \Big]$  and $C = \widetilde{\mathcal{F}}^* (I-\Mt^* \Mt)\BLambda (\BLambda^* \BLambda)^+ \BLambda^* \widetilde{\mathcal{F}}.$
\end{lemma}
\begin{proof} By \eqref{shrinkage-2}, we have
$\prox_f^{\tau}(q) =  S_{\tau}(q)  = (I-\widetilde{B}^+\widetilde{B})\big(q - \tau \mathcal{N}(q)\big),$
        thus $\refl_f^{\tau}(q)=2(I-\widetilde{B}^+\widetilde{B})\big(q - \tau \mathcal{N}(q)\big)-q.$
 By \eqref{prox-h} and $C$ in Lemma \ref{lem8}, we have
 $\prox_h^{\tau}(q) = C q +\widetilde{b},\quad  \widetilde{b}=\widetilde{\mathcal{F}}^*\Mt^* \Mt \BLambda M^*b.$ Hence,
 
        \begin{align*}
    H_{\tau}(q) &= \prox_h^{\tau}\big(\refl_f^{\tau}(q)\big) + q - \prox_f^{\tau}(q)\\
                &= C\big[ (I-2\Bt^+\Bt)q - 2\tau (I-\Bt^+\Bt)\mathcal{N}(q) \big] \\
                &\quad +
\widetilde{b} +\widetilde{B}^+\widetilde{B}q + \tau (I-\widetilde{B}^+\widetilde{B})\mathcal{N}(q),
\end{align*}
and
\begin{align*}
    & H_{\tau}(q) - H_{\tau}(q_*) \\
     =& \Big[C(I-\widetilde{B}^+\widetilde{B}) +(I-C)\widetilde{B}^+\widetilde{B} \Big] (q - q_*) \\
     & \quad + \tau \Big[ (I-2C)(I-\widetilde{B}^+\widetilde{B}) \Big] (\mathcal{N}(q) - \mathcal{N}(q_*)) \\ 
     =& \widetilde{H}(q - q_*) \\
     & \quad + \tau \Big[ (I-C)(I-\widetilde{B}^+\widetilde{B}) - C(I-\Bt^+\Bt) \Big] (\mathcal{N}(q) - \mathcal{N}(q_*)).
\end{align*}
This concludes the proof. \qed
\end{proof}

Notice that $C$ in Lemma \ref{lem8} is the projection matrix onto $\mathcal K \ker(A)$, and $I-C$ is the projection matrix onto $( \mathcal K^*)^{-1}\big[\Im (A^*)\big]$. Since $B$ and $\widetilde B$ in \eqref{eq:bTildeMatrix} are also projection matrices, we may rewrite them as follows. 
Define $C_0$ as the $2N \times(N-m)$ matrix whose columns form an orthonormal basis of $\mathcal K\ker (A)$, and $C_1$ the $2N \times (N+m)$ matrix whose columns form an orthonormal basis of $(\mathcal K^*)^{-1}\big[\Im(A^*)\big]$. Similarly we define the $2N\times 2(N-r)$ matrix $B_0$ and $2N\times 2r$ matrix  $B_1$ to be the matrices whose columns are orthonormal bases of $\ker(\Bt)$ and $\Im (\Bt^*)$ respectively. Therefore, we have 
\begin{equation}
\label{property-B-C}
    C_0C_0^*+C_1C_1^*=I,\quad B_0B_0^*+B_1B_1^*=I,
\end{equation}
and we can rewrite expression in Lemma \ref{lem8} as
\begin{align*}
    H_{\tau}(q) - H_{\tau}(q_*) = &\big[C_0C_0^*B_0B_0^* + C_1C_1^*B_1B_1^*\big](q-q_*) \\
    &+ \tau \big[ C_1C_1^*B_0B_0^* - C_0C_0^*B_0B_0^*\big](\mathcal{N}(q) - \mathcal{N}(q_*)). \notag
\end{align*}

{We note that the matrix {$C_0C_0^*B_0B_0^* + C_1C_1^*B_1B_1^*$} is similar to the matrix analyzed in \cite[Eq.~(2.5)]{demanet2016eventual}, which underlies their local convergence rate estimate. It will also play an important role in our estimate so we define 
\begin{equation}\label{eq:hTilde}
  {\widetilde{H} = C_0C_0^*B_0B_0^* + C_1C_1^*B_1B_1^*.}
\end{equation}
}

\begin{definition}  \cite{bjorck1973numerical}
     Let $\mathcal{U}$ and $\mathcal{V}$ be two subspaces of a linear space  with $\mathrm{dim}(\mathcal{U}) = s \le \mathrm{dim}(\mathcal{V})$. The principal angles $\theta_k \in [0,\frac{\pi}{2}]$ $(k=1, \ \dots,\  s)$ between $\mathcal{U}$ and $\mathcal{V}$, and principal vectors $u_j$ and $v_j$ are defined recursively as
    \begin{align*}
        \cos \theta_k &= \max_{u \in \mathcal{U}} \max_{v \in \mathcal{V}} \langle u, v \rangle = \langle u_k, v_k \rangle, \\
        & \|u\|  = \|v\|  = 1,    \ \langle u, u_j \rangle  = \langle v, v_j \rangle  = 0,  \ \forall j < k. 
    \end{align*}  
\end{definition}

Without loss of generality, assume $N-m \le 2(N-r)$. Let $\theta_i$  ($i=1,\  \hdots, \ N-m$) be the principal angles between $\mathcal K\ker(A)$ and $\ker(\Bt)$. Then $\theta_1 >0$ since $\mathcal K\ker(A) \cap \ker(\Bt) = \{0\}$ by Lemma \ref{thm:KNA_NB_intersection}. Let $\cos \Theta$ be the $(N-m)\times(N-m)$ diagonal matrix with diagonal entries $(\cos\theta_1,\ \hdots, \ \cos \theta_{N-m})$. By \cite[Theorem 1]{bjorck1973numerical}, the Singular Value Decomposition  (SVD) of the $(N-m)\times2(N-r)$ matrix $E_0 = C_0^*B_0$ is $E_0 = U_0\cos\Theta V^*$, with $V^*V =U_0^*U_0=U_0U_0^*= I_{(N-m)}$, and columns of $C_0U_0$ and $B_0V$ give the principal vectors. By the definition of SVD, $V$ is a matrix of size $2(N-r)\times(N-m)$, with orthonormal columns. Let $V'$ be a matrix of size $2(N-r)\times (N-2r+m)$ whose columns are normalized and orthogonal to columns of $V$. 
Define $\widetilde{V} = ( V | V')$, then $\widetilde{V} $ is a unitary matrix of size $2(N-r)\times 2(N-r)$. Now consider $E_1 = C_1^*B_0$, then by \eqref{property-B-C} we have
\begin{align*}
    E_1^*E_1 = B_0^*C_1C_1^*B_0 &= B_0^*B_0 - B_0^*C_0C_0^*B_0  = I_{(2N-2r)} - E_0^*E_0  \\
     &= I_{(2N-2r)} - V\cos^2\Theta {V}^* =\widetilde{V} \begin{pmatrix}
         \sin^2 \Theta & 0 \\
         0 & I_{(N-2r+m)}
     \end{pmatrix}\widetilde{V}^*,
\end{align*}
which implies   
  the SVD  $E_1 = U_1\begin{pmatrix}
         \sin \Theta & 0 \\
         0 & I_{(N-2r+m)}
     \end{pmatrix}\widetilde{V}^*$. Thus we have
\begin{align*}
    E_0E_0^*&=U_0\cos^2\Theta U_0^*, \quad E_1E_1^*= U_1\begin{pmatrix}
         \sin^2 \Theta & 0 \\
         0 & I_{(N-2r+m)}
     \end{pmatrix}U_1^*,\\
  E_1E_0^*&=U_1 \left( \begin{array}{ c}
          \sin\Theta\cos\Theta  \\ 0   
    \end{array} \right)  U_0^*,\quad  E_0E_1^*=U_0\left( \begin{array}{ c|c}
          \cos\Theta \sin\Theta & 0 \\  
    \end{array} \right) U_1^*.
\end{align*} 
Notice that $B_0=(C_0C_0^*+C_1C_1^*)B_0=C_0E_0+C_1E_1$, so we obtain 
\begin{align*}
    B_0B_0^* &=(C_0 | C_1 ) \left( \begin{array}{c|c}
        E_0E_0^* & E_0E_1^* \\ \hline
        E_1E_0^* &  E_1E_1^*
    \end{array} \right) (C_0| C_1)^* \\
    &= (C_0U_0 | C_1 U_1) \left( \begin{array}{c|cc}
        \cos^2\Theta & \cos\Theta \sin\Theta & 0 \\ \hline
        \sin \Theta \cos\Theta  & \sin^2\Theta & 0\\
        0 & 0 & I_{(N-2r+m)} 
    \end{array} \right) (C_0U_0 | C_1U_1)^*.
\end{align*}
Define $\Ct_0 = C_0U_0$ and $\Ct_1 = C_1U_1$, which are $2N\times (N-m)$ and $2N\times 2(N-r)$ matrices respectively. Then the columns of $\Ct_0$ form an orthonormal basis of $\mathcal K\ker(A)$, and columns of $\Ct_1$ are orthonormal vectors in $(\mathcal K^*)^{-1}\big[\Im(A^*)\big]$. Let $D$ denote the orthogonal complement of $(\mathcal K^*)^{-1}\big[\Im(A^*)\big] \cap \Im (\Bt^*)$ in the subspace $(\mathcal K^*)^{-1}\big[\Im(A^*)\big]$. Then $\mathrm{dim}(D)=2(N-r)$. Since $\theta_1 >0$, the largest angle between $(\mathcal K^*)^{-1}\big[\Im(A^*)\big]$ and $\ker(\Bt)$ is less than $\pi/2$. So none of the column vectors of $\Ct_1$ are orthogonal to $\ker(\Bt)$. Hence, by counting the dimension of $D$ and columns of $\Ct_1$, we conclude that columns of $\Ct_1$ form an orthonormal basis of $D$. Define $\Ct_2$ to be the $2N\times (m+2r-N)$ matrix whose columns form an orthonormal basis of $\Im (\Bt^*) \cap (\mathcal K^*)^{-1}\big[\Im(A^*)\big]$. Then,
\begin{align}\label{eq:B0}
    B_0B_0^* = (\Ct_0 | \Ct_1 | \Ct_2) \left( \begin{array}{c|cc|c}
        \cos^2\Theta & \cos\Theta \sin\Theta & 0  & 0\\ \hline
        \sin \Theta \cos\Theta  & \sin^2\Theta & 0 & 0\\
        0 & 0 & I_{(N-2r+m)} & 0 \\ \hline
        0 & 0 & 0 & 0
    \end{array} \right) (\Ct_0 | \Ct_1 | \Ct_2)^*.
\end{align}
By \eqref{property-B-C}, $ B_1B_1^* =I- B_0B_0^*$, thus
\begin{align}\label{eq:B1}
    B_1B_1^* = (\Ct_0 | \Ct_1 | \Ct_2) \left( \begin{array}{c|cc|c}
        \sin^2\Theta & -\cos\Theta \sin\Theta & 0  & 0\\ \hline
        -\sin \Theta \cos\Theta  & \cos^2\Theta & 0 & 0\\
        0 & 0 & 0 & 0 \\ \hline
        0 & 0 & 0 & I_{(m+2r-N)}
    \end{array} \right) (\Ct_0 | \Ct_1 | \Ct_2)^*.
\end{align}

\begin{sloppypar}
{Notice that $C_0C_0^*$ and $C_1C_1^*$ are, respectively, the projection matrices to the spaces $\mathcal K \ker(A)$ and $(\mathcal K^*)^{-1}\big[\Im(A^*)\big]$. Therefore, since the columns of $\Ct_0$ lie in $\mathcal K \ker(A)$, whereas the columns of $\Ct_1$ and $\Ct_2$ lie in $(\mathcal K^*)^{-1}\big[\Im(A^*)\big]$, we obtain 
\begin{equation}\label{c0_ortho}
    C_0C_0^*(\Ct_0 | \Ct_1 | \Ct_2) = (\Ct_0 | 0 | 0), \quad \ C_1C_1^*(\Ct_0 | \Ct_1 | \Ct_2) = (0 | \Ct_1 | \Ct_2).
\end{equation}}
\end{sloppypar}

{Now start from \eqref{eq:hTilde} and use equations \eqref{eq:B0}, \eqref{eq:B1}, and \eqref{c0_ortho} to obtain  }
\begin{align} 
   \Ht
   &= (\Ct_0 | \Ct_1 | \Ct_2) \left( \begin{array}{c|cc|c}
        \cos^2\Theta & \cos\Theta \sin\Theta & 0  & 0\\ \hline
        -\sin \Theta \cos\Theta  & \cos^2\Theta & 0 & 0\\
        0 & 0 & 0 & 0 \\ \hline
        0 & 0 & 0 & I_{(m+2r-N)}
    \end{array} \right) (\Ct_0 | \Ct_1 | \Ct_2)^*.\label{eq:H_op}
\end{align}
{Similarly, one can derive}
\begin{align*}
    &C_1C_1^*B_0B_0^* - C_0C_0^*B_0B_0^* \\
    =& (\Ct_0 | \Ct_1 | \Ct_2) \left( \begin{array}{c|cc|c}
        -\cos^2\Theta & -\cos\Theta \sin\Theta & 0  & 0\\ \hline
        \sin \Theta \cos\Theta  & \sin^2\Theta & 0 & 0\\
        0 & 0 & I_{(N-2r+m)} & 0 \\ \hline
        0 & 0 & 0 & 0
    \end{array} \right) (\Ct_0 | \Ct_1 | \Ct_2)^*.
\end{align*}
  We now summarize the discussion above as the following result.
\begin{lemma}\label{lem-DRS-expression}
   For any $q\in \mathcal{Q}$ 
    and any DRS fixed point $q_*$,    
    \begin{align*}
     & H_{\tau}(q) - H_{\tau}(q_*) =\Ct \left( \begin{array}{c|cc|c}
        \cos^2\Theta & \cos\Theta \sin\Theta & 0  & 0\\ \hline
        -\sin \Theta \cos\Theta  & \cos^2\Theta & 0 & 0\\
        0 & 0 & 0 & 0 \\ \hline
        0 & 0 & 0 & I_{(m+2r-N)}
    \end{array} \right) \Ct^*(q-q_*)  \\ 
    &+\Ct  \left( \begin{array}{c|cc|c}
        -\cos^2\Theta & -\cos\Theta \sin\Theta & 0  & 0\\ \hline
        \sin \Theta \cos\Theta  & \sin^2\Theta & 0 & 0\\
        0 & 0 & I_{(N-2r+m)} & 0 \\ \hline
        0 & 0 & 0 & 0
    \end{array} \right)\Ct^* \tau\big[ \mathcal{N}(q) - \mathcal{N}(q_*)\big],
\end{align*}
\begin{sloppypar}\noindent
where $\Ct = (\Ct_0|\Ct_1 | \Ct_2)$. $\Ct_0$, $\Ct_1$, and $\Ct_2$ are the matrices whose columns form an orthonormal basis of $\mathcal K\ker(A)$, $D$, and $\Im (\Bt^*) \cap (\mathcal K^*)^{-1}\big[\Im(A^*)\big]$ respectively, where $D$ is the orthogonal complement of $\Im (\Bt^*) \cap (\mathcal K^*)^{-1}\big[\Im(A^*)\big]$ in $(\mathcal K^*)^{-1}\big[\Im(A^*)\big]$.
\end{sloppypar}
\end{lemma}

\begin{lemma}
\label{C2-subspace-Lemma}
    Assume   DRS iterates $q_k$ converge to an interior fixed point $q_*$. Then there exists $K \in \mathbb{N}$ such that for all $k\ge K$, 
     $\mathbb{P}(q_k- q_*) = 0,$
    where $\mathbb{P}$ is the Euclidean projection to $\Im(\widetilde{B}^*) \cap (\mathcal K^*)^{-1}[\Im(A^*)]$.
\end{lemma}
\begin{proof}
Since $q_*$ is in the interior of $\mathcal Q$, there exists $K$ such that $q_k \in \mathcal Q$ for all $k\geq K$.
 Since columns of $\Ct_2$  span the subspace $\Im(\widetilde{B}^*) \cap (\mathcal K^*)^{-1}[\Im(A^*)]$,   Lemma \ref{lem-DRS-expression} and Lemma \ref{lem8} imply that $ \mathbb{P}(q_k- q_*) = \mathbb{P}(q_K- q_*)$ for all $k\geq K$.
 If $\mathbb{P}(q_K- q_*)\ne 0$,
 then  $\mathbb{P}(q_k- q_*)$ is a constant {vector} for $k\geq K$, which contradicts $q_k\to q_*$. So $\mathbb{P}(q_K- q_*)=0$, which implies $\mathbb{P}(q_k- q_*)=0$ for any $k\geq K$. \qed 
\end{proof}

\vspace{-0.5cm}
\subsection{The Proof of the Main Theorem}

Now we are ready to prove Theorem \ref{thm-main}.

\begin{proof}

First of all, by Definition \ref{def:qSpace} and Lemma \ref{lem6}, any DRS fixed point is in the set $\mathcal Q$. 
The convexity of the problem \eqref{eq:doubleDual}
ensures that DRS iterates converge to the minimizer $v_*$, i.e.,  $q_k$ converges to some fixed point $q_*$ to DRS and $S_\tau (q_*)=\prox_f^\tau (q_*)=v_*$. 
For a standard problem, $q_*$ {belongs to the interior of the set $\mathcal Q$}. We first discuss a simple case that  $\Im(\Bt^*) \cap (\mathcal K^*)^{-1}\big[\Im(A^*)\big] = \{0\} $. By Lemma \ref{lem6} and the definition of $\widetilde{C}_2$, we deduce that the fixed point is unique and $m + 2r = N$.
  Notice   \eqref{eq:H_op} shows that $\| \widetilde{H} \|_2 = \cos \theta_1$, where $\|\widetilde{H}\|_2$ denotes the matrix spectral norm. 
For any $q\in \mathcal Q$,   by the fact that $C$ is a projection matrix and  Lemma \ref{lem7}, we have $\|[ (I-2C)(I-\widetilde{B}^+\widetilde{B}) ] (\mathcal{N}(q) - \mathcal{N}(q_*))\|\leq \max_{j:\|(v_*)_j\|\ne 0}\frac{2}{\|(v_*)_j\|} \|q-q_*\|$, thus by the triangle inequality
\begin{align*}
    \|H_{\tau}(q) - H_{\tau}(q_*)\| &\le \|\widetilde{H}\|_2 \|q-q_*\| + \tau \|[ (I-2C)(I-\widetilde{B}^+\widetilde{B}) ] (\mathcal{N}(q) - \mathcal{N}(q_*))\| \\
     &\leq \left(\cos\theta_1+ \max_{j:\|(v_*)_j\|\ne 0}\frac{2\tau}{\|(v_*)_j\|} \right)\|q - q_* \|.
\end{align*}
Since $q_k$ converges to $q_*$ and $q_*$ is in the interior of $\mathcal Q$,
there exists $K$ such that for all $k\ge K$, $q_k\in \mathcal Q$.
Hence, there exists $K$ such that for all $k\ge K$, we have 
$$\|H_{\tau}(q_k) - H_{\tau}(q_*)\| \le \left(\cos(\theta_1) + \max_{j:\|(v_*)_j\|\ne 0}\frac{2\tau}{\|(v_*)_j\|}\right)^{k-K}\|q_K - q_*\|.$$ 

Now we consider the case when $\Im(\Bt^*) \cap (\mathcal K^*)^{-1}\big[\Im(A^*)\big] \ne \{0\} $ for which DRS fixed points are not unique by Lemma \ref{lem6}.   By Lemma \ref{C2-subspace-Lemma}, there exists $K$  such that $\mathbb{P}_{\Im(\widetilde{B}^*) \cap (\mathcal K^*)^{-1}[\Im(A^*)]}(q_k- q_*) = 0,\  \forall k\geq K.$ 
Since $\widetilde{C}_2$ is the basis for $\Im(\widetilde{B}^*) \cap (\mathcal K^*)^{-1}[\Im(A^*)]$, the decomposition in \eqref{eq:H_op} implies  $\|\widetilde{H}(q_k-q_*)\| \le \cos \theta_1 \|q_k-q_*\|$ for $k\geq K$, thus
\begin{align*}
   \|H_{\tau}(q_k) - H_{\tau}(q_*)\| &\le \|\widetilde{H}(q_k-q_*)\| \\
   & \quad + \tau \|[ (I-2C)(I-\widetilde{B}^+\widetilde{B}) ] (\mathcal{N}(q_k) - \mathcal{N}(q_*))\| \\ 
     &\leq  \left(\cos\theta_1+ \max_{j:\|(v_*)_j\|\ne 0}\frac{2\tau}{\|(v_*)_j\|} \right)\|q_k - q_* \|.
\end{align*}
This concludes the proof.\qed
\end{proof}

\begin{remark}
{ Both $\theta_1$ and $v_*$ depend on the minimizer, 
thus it is in general very difficult to find an a priori range of $\tau$ such that 
$$ \cos\theta_1 + \max_{j:\|(v_*)_j\|\ne 0} \frac{2\tau}{\|(v_*)_j\|_2} <1,$$
although a very small $\tau>0$ suffices.
In our numerical tests reported, we first find a numerical minimizer by performing many iterations of ADMM. Then, we compute its $\theta_1$ and $v_*$, with which we plotted figures such as Figure \ref{fig:intro-linear}. }
\end{remark}

\subsection{Possible extensions}

{The more general DRS operator can be written as $H^\lambda_\tau=(1-\lambda) \iden + \lambda\frac{\iden+\refl_h^\tau\refl_f^\tau }{2}$ with a relaxation parameter $\lambda\in (0,2)$. Since $H^\lambda_\tau$ is very similar to $H_\tau$, it is straightforward to extend Theorem \ref{thm-main} for generalized ADMM. An outline of the argument is provided below. \begin{lemma}\label{lem_generalized}
    The set of fixed points of the generalized DRS iteration operator, $H^{\lambda}_\tau = (1-\lambda)I + \lambda \frac{I + \refl_h^\tau \refl_f^\tau}{2}$,
for the problem~\eqref{eq:doubleDual} coincides with the set of fixed points of the standard DRS iteration:
\[
\big\{\, v_* + \tau \eta : \eta \in \partial f(v_*) \cap (\mathcal K^*)^{-1}[\Im(A^*)] \,\big\}.
\]
Moreover, the fixed point is unique if and only if
\[
(\mathcal K^*)^{-1}[\Im(A^*)] \cap \Im(\widetilde B^{\,T}) = \{0\},
\]
where $\widetilde B^{\,T}$ is the transpose of the matrix $\widetilde B$ defined in~\eqref{eq:bTildeMatrix}.
\end{lemma} 
\begin{proof}
From the proof of Lemma \ref{lem6}, for any $\eta \in \pd f(v_*) \cap(\mathcal K^*)^{-1}\big[\Im (A^*) \big]$, then $S_{\tau}(v_* + \tau \eta) = v_*$, $\prox_h^{\tau}(v_*-\tau \eta) = v_* $, and $\refl_f^{\tau}(v_*+\tau \eta) = v_*-\tau \eta$. Hence, 
\begin{align*}
    H_{\tau}^{\lambda}(v_*+\tau \eta) &= v_*+\tau \eta + \lambda \left[  \prox_h^{\tau}(\refl_f^{\tau}(v_*+\tau \eta)) - \prox_f^{\tau}(v_*+\tau \eta) \right] \\ 
    &= v_*+\tau \eta + \lambda \left[ \prox_h^{\tau}(v_*-\tau \eta) - v_*  \right] = v_* +\tau \eta.
\end{align*}
Conversely, if $H_{\tau}^{\lambda}(q) = q$ and $\eta = \frac{q-v_*}{\tau}$, then once again by the convergence of the generalized DRS iteration \cite{lions1979splitting}, $\prox_f^{\tau}(q) = v_*,$ thus $\eta \in \partial f(v_*)$. Secondly, $H_{\tau}^{\lambda}(q) = q$ implies that {$v_* = \prox_h^{\tau}(2v_*-q)$}, which shows $\eta \in (\mathcal K^*)^{-1}\big[\Im (A^*)\big]$. For the discussion on uniqueness, reference the proof in Lemma \ref{lem6}.
\end{proof} 
The next two Lemmas are the generalized versions of Lemmas \ref{lem8} and \ref{lem-DRS-expression}.
\begin{lemma}\label{lem-gDRS-expression}
    For any $q\in\mathcal{Q}$ and any generalized DRS fixed point $q_*$,
    \begin{align*}
        H^{\lambda}_{\tau}(q) - H^{\lambda}_{\tau}(q_*) = \widetilde{H}^{\lambda}(q-q_*) +\tau \lambda \left[ (I-2C)(I-\Bt^+\Bt)\right]( \mathcal{N}(q) - \mathcal{N}(q_*)),
    \end{align*}
    where $\widetilde{H}^{\lambda} = (1-\lambda) I + \lambda \widetilde{H}$, and $\widetilde{H}$ and $C$ are defined in Lemma \ref{lem8}.
\end{lemma}
\begin{proof}
    The proof follows the same steps as the proof of Lemma~\ref{lem8}, applied to the operator $H_{\tau}^{\lambda}$.
\end{proof}
\begin{lemma}\label{lem-grs_struct}
{\small\setlength{\arraycolsep}{1pt}
   For any $q\in \mathcal{Q}$ 
    and any generalized DRS fixed point $q_*$,
    \begin{align*}
    & H_{\tau}^{\lambda}(q) - H_{\tau}^{\lambda}(q_*) = \\& \Ct \left( \begin{array}{c|cc|c}
        \cos^2\Theta {+} (1{-}\lambda)\sin^2\Theta & \lambda\cos\Theta  \sin\Theta & 0  & 0\\ \hline
        -\lambda \sin \Theta \cos\Theta  & \cos^2\Theta {+} (1{-}\lambda)\sin^2\Theta & 0 & 0\\
        0 & 0 & (1{-}\lambda) I_{(N-2r+m)} & 0 \\ \hline
        0 & 0 & 0 & I_{(m+2r-N)}
    \end{array} \right) \Ct^*(q-q_*)  \\ 
    &+\Ct  \left( \begin{array}{c|cc|c}
        -\cos^2\Theta & -\cos\Theta \sin\Theta & 0  & 0\\ \hline
        \sin \Theta \cos\Theta  & \sin^2\Theta & 0 & 0\\
        0 & 0 & I_{(N-2r+m)} & 0 \\ \hline
        0 & 0 & 0 & 0
    \end{array} \right)\Ct^* \lambda\tau\left( \mathcal{N}(q) - \mathcal{N}(q_*)\right),
\end{align*}
}
\begin{sloppypar}\noindent
where $\Ct = (\Ct_0|\Ct_1 | \Ct_2)$. $\Ct_0$, $\Ct_1$, and $\Ct_2$ are the matrices whose columns form an orthonormal basis of $\mathcal K\ker(A)$, $D$, and $\Im (\Bt^*) \cap (\mathcal K^*)^{-1}\big[\Im(A^*)\big]$ respectively, where $D$ is the orthogonal complement of $\Im (\Bt^*) \cap (\mathcal K^*)^{-1}\big[\Im(A^*)\big]$ in $(\mathcal K^*)^{-1}\big[\Im(A^*)\big]$.
\end{sloppypar}
\end{lemma}
\begin{proof}
    The proof follows from Lemma \ref{lem-DRS-expression} and Lemma \ref{lem-gDRS-expression}.
\end{proof}
Notice that $\widetilde{H}^{\lambda}$ is a normal matrix and for any $\lambda \in (0,2),$ $ \cos\theta_1 \le \| \widetilde{H}^{\lambda}\|_2 = \sqrt{\lambda(2-\lambda)\cos^2\theta_1 + (1-\lambda)^2} < 1.$ Therefore, we can state the following theorem for generalized ADMM.
\begin{theorem}
\label{thm-main_gen}
    Let $\theta_1$ be the smallest non-zero principal angle  between the two linear spaces $\mathcal{K}\ker{A}= \{ \mathcal{K}u: u \in \ker(A)\}$ and $\ker(\widetilde{B})$ with $\widetilde{B}$ defined in \eqref{eq:bTildeMatrix}. 
    Consider generalized ADMM solving \eqref{eq:tvcs_primal-1} with a step size $\gamma=\frac{1}{\tau}>0$, which is equivalent to generalized DRS solving \eqref{eq:doubleDual} with a step size $\tau$. 
    The convexity of the problem \eqref{eq:doubleDual} implies that  generalized DRS iterates $q_k$ converge to a fixed point $q_*$.
    Assume that   $q_*$ is  an interior fixed point.
    Under Assumption \ref{asmp:sizeOmega}, with probability $1-O(N^{-s})$,  for small enough $\tau>0$, there is an integer $K$ such that for all $k\ge K$, 
    \begin{align*} \|q_k - q_*\| \le \left[\|\widetilde{H}^{\lambda} \|_2 + \max_{j:\|(v_*)_j\|\ne 0} \frac{2\tau\lambda}{\|(v_*)_j\|_2} \right]^{k-K} \|q_K - q_*\|,\end{align*} where $\|\widetilde{H}^{\lambda}\|_2 \in [\cos\theta_1,1)$ for all $\lambda \in (0,2).$
 \end{theorem}
}

It is possible to extend the discussion to more general problems and algorithms, but we do not pursue these extensions. The following extensions can be considered:
\begin{enumerate}
\item  In Theorem \ref{thm-main}, we only considered the case that $q_*\in \mathcal Q$ lies in the interior of $\mathcal Q$. For the non-standard cases, iterates $q_k$ converge to a fixed point lying on the boundary of the set $\mathcal Q$, and it is possible to have a similar result with a redefined angle $\theta_1$ following the arguments for such non-standard cases in \cite{demanet2016eventual}.  Whether the converged fixed point is standard or non-standard depends on the data $(A,b)$ and initial guess $q_0$ of the DRS iteration. In our numerical tests,  we have not observed non-standard cases.
     
    \item One can also consider adding regularization \cite{friedlander2008exact} to problem \eqref{eq:tvcs_primal-1}. One suitable way of adding regularization is to add $\ell^2$ regularization to the equivalent problem \eqref{eq:doubleDual} with parameter $\alpha$. 
     \begin{equation}\label{eq:DD_regularize}
       \min_{v\in \R^{N\times d}} \|v\|_{1,2}+\iota_{\mathcal K \{u: Au=b\}}(v)+
       \frac{1}{2\alpha}\|v\|^2,
 \end{equation}
  where $\|\cdot\|$ is the 2-norm for $\R^{Nd}$.
 When $\alpha$ is large enough, \eqref{eq:DD_regularize} gives the same minimizer \cite{yin2010analysis}.
{It is still possible to characterize the set of fixed points for the DRS operator, $H^{\lambda, \alpha}_{\tau}$, corresponding to \eqref{eq:DD_regularize}. We can also find the expression, \begin{align*} & H^{\lambda, \alpha}_{\tau}(q) - H^{\lambda, \alpha}_{\tau}(q_*) \\ & \quad = \widetilde{H}^{\lambda,\alpha}(q-q_*) + \frac{\tau\lambda \alpha}{\alpha + \tau} (I-2C)(I-\Bt^+\Bt)\left(\mathcal{N}(q) - \mathcal{N}(q_*) \right),\end{align*} for any $q\in\mathcal{Q}$ and generalized DRS fixed point $q_*$. However, the matrix $\widetilde{H}^{\lambda,\alpha}$ is no longer normal, which causes additional difficulties in our current argument for estimating the rate, due to the extra nonlinear operator $\mathcal N$. Some new techniques are needed to prove a local linear rate, which will be our future work.}
\end{enumerate}
 
\section{Numerical Tests}
\label{sec:tests}

We report numerical results of implementing Algorithm \ref{alg-Gprox-PDHG-2} with step sizes $\sigma=\frac{1}{\tau}$ for solving the TVCS problem \eqref{eq:tvcs_primal-1} formulated as in \eqref{eq:primala}, which is equivalent to ADMM, with step size $\gamma=\frac{1}{\tau}$ on \eqref{eq:primala} by the relations in Table \ref{tab:relations}. We construct TVCS problems using 2D and 3D Shepp-Logan images \cite{shepp1974fourier} as well as {3D MRI data of the human brain provided with the consent of individual(s) who wished to remain unacknowledged. More information on this data can be found in the Data Acknowledgement section.} The 2D tests were performed on a MacBook Air with M1 Chip (8 core) with 16GB memory, while the 3D tests were performed on one Nvidia A100 GPU card with {80GB} memory, implemented in Python with single and double precision.
 Similar to \cite{liuGPU}, the Python package JAX was used to achieve a simple implementation on the GPU. 
  Unless stated otherwise, the initial conditions used for all the tests were the given data $u_0 =  {\mathcal{F}}^* M^* M  {\mathcal{F}} u_*$ and $p_0 = 0$, where $u_*$ is the true image (Shepp-Logan or MR image). The mask matrix $M$ is generated  randomly. {The primary goal of the numerical tests is to validate the linear convergence rate for 2D and 3D TVCS problems. In addition, the experiments aim to examine how the step size \(\tau\) influences the local convergence rate in practice, study the effects of the regularization and relaxation parameters, evaluate the performance on real MRI data of the human brain, and investigate how single- and double-precision arithmetic impact both the convergence rate and the computational time.
}

\subsection{2D Shepp-Logan Image}
 We first study how sharp the estimate in Theorem \ref{thm-main} is for small $\tau$. We construct a test problem for TVCS starting out with a 2D Shepp-Logan image \cite{shepp1974fourier}  where we randomly observe  $30 \%$ of the frequencies, including the zeroth frequency.
\subsubsection{Local Linear Rate Validation}
 Figure \ref{fig:intro-linear} (left) shows result $\tau=0.01$   for 2D image of size $64\times 64$.
For computing the angle $\theta_1$, we need the minimizer $v_*$, to \eqref{eq:doubleDual}, which is approximated by running $10000$ iterations of ADMM on \eqref{eq:tvcs_primal-1} and then using Table \ref{tab:relations} to transform the ADMM variables into the physical variable for DR on \eqref{eq:doubleDual}.
The angle between the subspaces  $\mathcal{K}\ker(A)$ and $\ker(\Bt)$  is then computed by SVD.
In Figure \ref{fig:intro-linear} (left), we observe that  $\cos\theta_1$ matches quite well with the actual local linear rate. The estimate in Theorem \ref{thm-main} is more conservative, but for  $\tau=0.01$ it still seems a good estimate in practice. 
On the other hand, the linear convergence regime is not reached until iteration number 4300, and the number of iterations needed to enter the linear convergence regime can be sensitive to $\tau$ in practice.  A larger $\tau$ may give fewer iterations needed to enter the linear convergence regime \cite{liang2017local}. 
 
\subsubsection{The Effects of Different step size $\tau$}\label{sec:tests_tau}
For the same 2D problem,  Figure \ref{fig:stepSize_comp} shows the results for different step sizes ranging from $\tau = 0.01$ to $\tau = 20$, which do not induce a big change in the local linear rate, even though our provable rate does contain $\tau$ in the estimate.
  We remark that the dependence on $\tau$ in Theorem \ref{thm-main} can be removed in our proof when $d=1$, i.e., the local linear rate of Algorithm \ref{alg-ADMM} for $\ell^1$-norm CS problem does not depend on step size in both analysis and numerical tests \cite{demanet2016eventual}.
On the other hand, Figure \ref{fig:stepSize_comp} shows that different step sizes significantly affect the number of iterations needed to enter the linear convergence regime. As shown in Figure \ref{fig:stepSize_comp}, for $\tau=20.0$, the number of iterations it takes to enter linear convergence regime is $l=0$; that is, numerically, this suggests global linear convergence.

\begin{figure}[htbp!]
    \centering
    \includegraphics[width=0.5\textwidth]{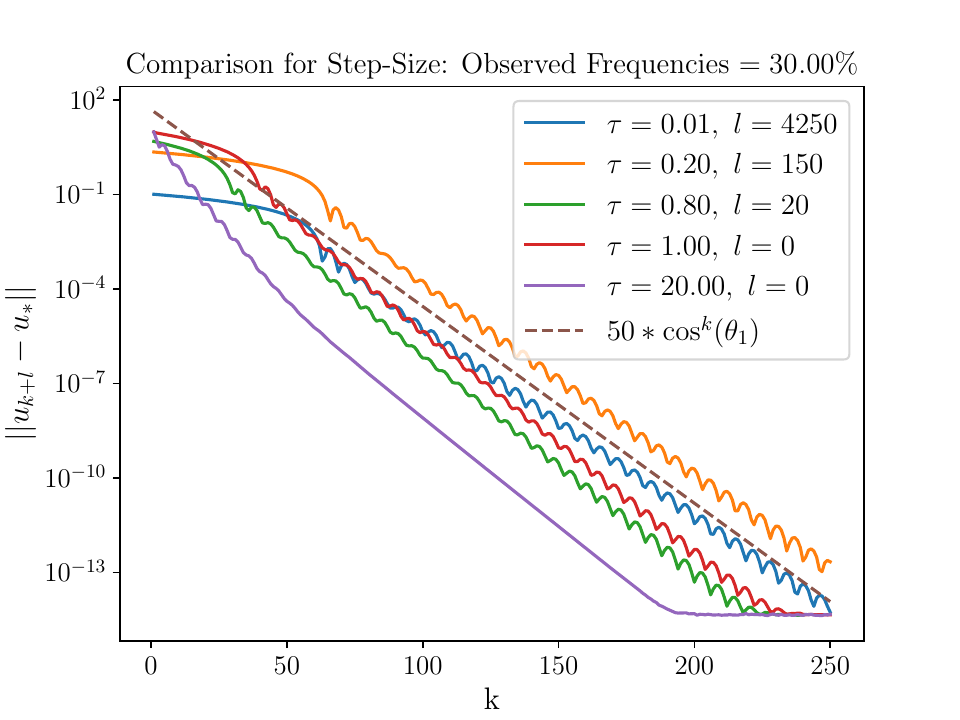}
    \caption{Algorithm \ref{alg-ADMM} with $\gamma = \frac{1}{\tau}$ for \eqref{eq:tvcs_primal-1} with $30 \%$ observed frequencies for  a 2D Shepp-Logan image of size $64\times 64$. {\bf Here $k$ is not the iteration number. Instead, $k+l$ is the iteration number where $l$ is the number of iterations needed to enter the linear convergence regime. }}
    \label{fig:stepSize_comp}
\end{figure}

\subsection{Effects of Regularization and Relaxation}\label{sec-regularization}

Consider a generalized version of ADMM by applying the general DRS operator  
 $H^\lambda_\tau=(1-\lambda) \iden + \lambda\frac{\iden+\refl_h^\tau\refl_f^\tau }{2}$ with a relaxation parameter $\lambda\in (0,2)$
to the regularized problem \eqref{eq:DD_regularize} with a regularization parameter $\alpha$.
See Figure \ref{fig:linear_regularized} for results with different $\lambda$ and  $\alpha = 100$. For these tests, the 2D Shepp-Logan image is $128 \times 128$, the step size is $\gamma = \frac{1}{\tau} = \frac{1}{22}$, and $30\%$ of the frequencies are observed. As proven in \cite{demanet2016eventual},   special choices of parameters $\alpha$ and $\lambda$ can speed up the local linear convergence rate for $\ell^1$-norm CS problem. Figure \ref{fig:linear_regularized} shows that this is also the case for TVCS in two dimensions.

\begin{figure}[tbhp!]
    \centering
    \includegraphics[width=0.5\textwidth]{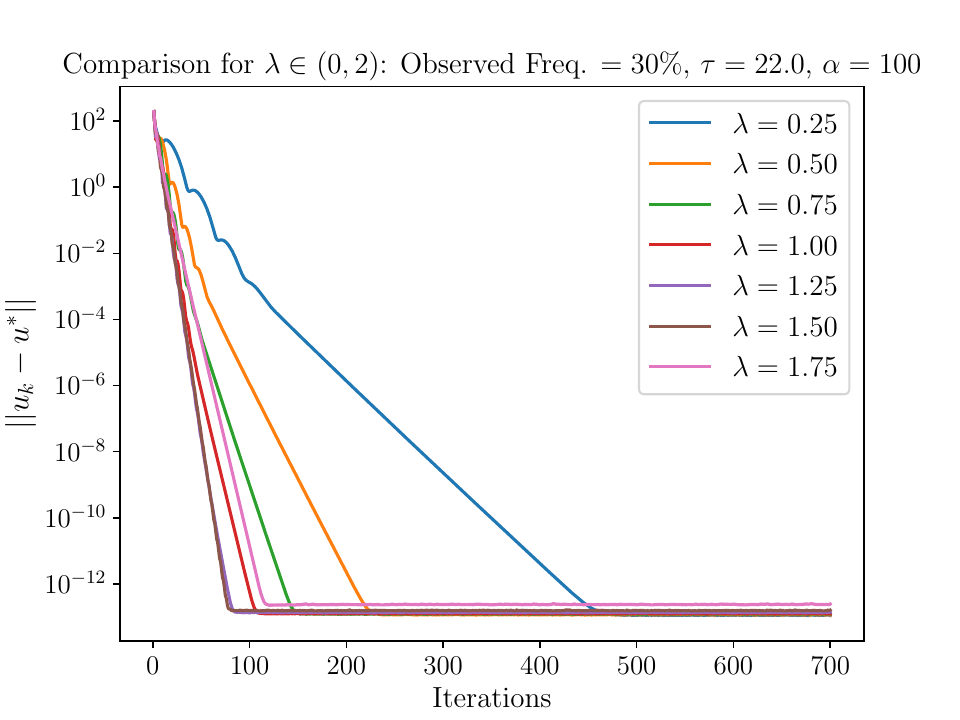}
    \caption{Local convergence rate of generalized ADMM (corresponding to the general DRS operator $H_{\tau}^{\lambda}$) solving \eqref{eq:DD_regularize} with different parameters $\lambda$, $\alpha = 100$. A 2D Shepp-Logan image of size $128\times 128$ with $30 \%$ observed frequencies. }
    \label{fig:linear_regularized}
\end{figure}

\subsection{3D Images}

 \begin{figure}[htbp]
    \centering
    \includegraphics[width=0.8\textwidth]{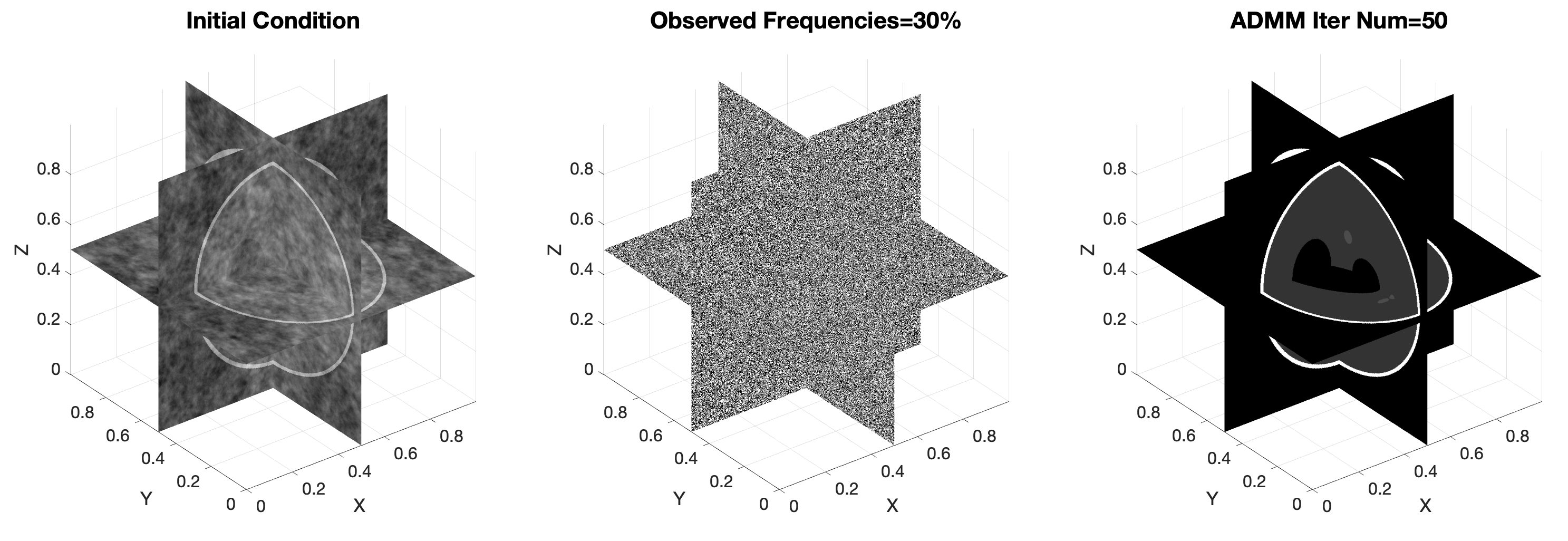}
    \caption{ Left: The initial guess in ADMM. Middle: the non-zero entries of the mask, observed frequencies are $30\%$. Right: the primal iterate output by ADMM after $50$ iterations $u_{50}$ for a 3D Shepp-Logan image
    of size $512^3$.}
    \label{fig:3dsheppLogan}
\end{figure}

\begin{figure}[tbhp!]
    \centering
    \includegraphics[width=0.8\textwidth]{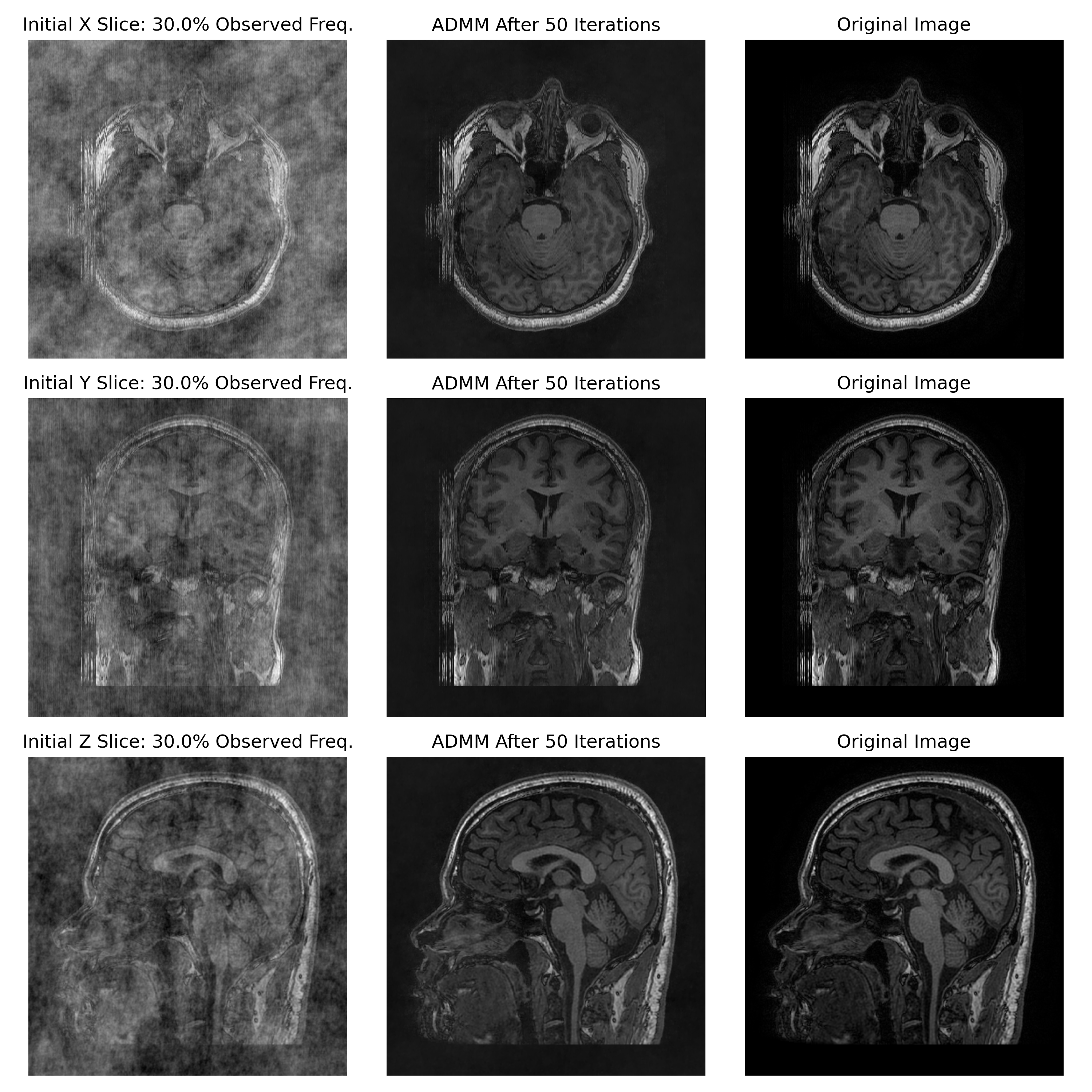}
    \caption{3D MR image of size $512^3$. Left: slices of the initial condition for the primal variable in ADMM. Middle: slices of the primal variable of ADMM with $\gamma = \frac{1}{22}$ after $50$ iterations. Right: slices of the true MR image.  }
    \label{fig:mri_freq70}
\end{figure}

\begin{table}[ht]
 \caption{GPU Time (minutes) vs. relative error $\left(\frac{\|u_k-u_*\|}{\|u_*\|}\right)$ of ADMM (Algorithm \ref{alg-ADMM} with step size $\gamma = \frac{1}{\tau}=\frac{1}{22}$) solving \eqref{eq:tvcs_primal-1} with $30 \%$ observed frequencies for real MRI data of size $512^3$. Double precision computation in Python JAX on one Nvidia A100 card with 80GB memory.}
    \centering
    \begin{tabular}{|c|c|c|c|c|c|}
       
        \hline
        Iteration Number & 1 & 10 & 20 & 80 & 350\\ \hline
        GPU Time (min) & $0.02$ & $0.06$ & $0.1$ & $0.33$ & $1.23$  \\ \hline
        Relative Error & $6.2\times 10^{-1}$ & $2.9\times 10^{-2}$ & $7.2\times 10^{-3}$ & $8.7\times 10^{-4}$ & $9.5\times 10^{-5}$ \\ \hline
    \end{tabular}
    \label{tab:gpuVsError}
\end{table}

\begin{figure}[tbhp!]
    \centering
    \includegraphics[width=0.45\textwidth]{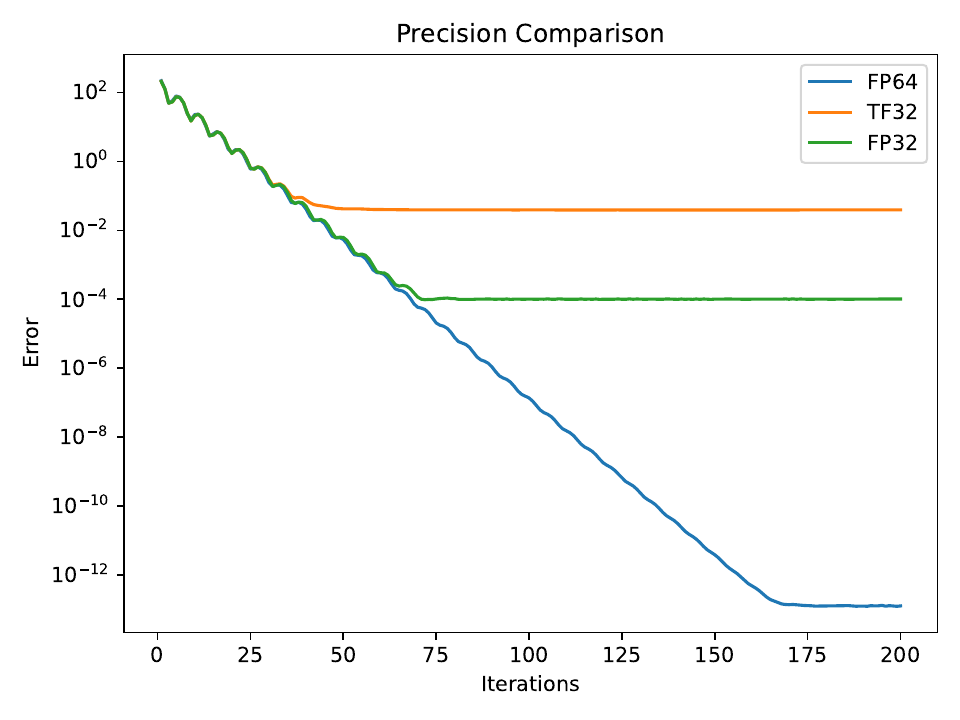}
    \caption{3D  Shepp-Logan image of size $128^3$ with $30\%$ observed frequencies. ADMM with $\gamma = \frac{1}{\tau} = \frac{1}{22}$. Performance of the algorithm using Double Precision vs. Single Precision (FP32 and TF32) in Python JAX on Nvidia A100. }
    \label{fig:precision_comparison}
\end{figure}

\begin{table}[htbp!]
\caption{Comparison of the computational time (in seconds) of Algorithm \ref{alg-ADMM} with $\gamma = \frac{1}{\tau}=\frac{1}{22}$ to perform 250 iterations of ADMM implemented by Python JAX on one Nvidia A100 80GB card: double-precision (FP64) vs. single-precision (FP32 and TF32). 3D Shepp-Logan of different sizes with $30\%$ observed frequencies. For FP64, memory is not sufficient to compute the problem size $700^3$.  }
    \centering
    \begin{tabular}{|c|c|c|c|c|}
        \hline
         Problem Size & $128^3$ & $256^3$ & $512^3$ & $700^3$  \\ \hline
        FP64 & 2.36 & 7.70 & 57.14 & -  \\ \hline
        FP32 & 2.34 & 4.78 & 31.14 & 86.98 \\ \hline
        TF32 & 2.30 & 4.30 & 23.79 & 62.22  \\ \hline
    \end{tabular}

    \label{tab:example}
\end{table}

We test large 3D problems using the 3D  Shepp-Logan image as well as some MRI data with $30\%$ observed frequencies. The step size is taken to be $\tau=0.1$. An estimate of $v_*$ was obtained by running ADMM on \eqref{eq:tvcs_primal-1} for 10,000 iterations and then using the relations in Table \ref{tab:relations} to obtain the physical variable of DR $v_k$. The angle between two subspaces is approximated by the procedure in  \cite[Appendix B]{demanet2016eventual}.

First, we consider a  3D  Shepp-Logan image of size $512^3$, and the performance is shown in  Figure \ref{fig:intro-linear} (right) and also Figure \ref{fig:3dsheppLogan}.
Next we verify the performance on some MR images of size $512^3$ with $30 \%$ frequencies observed. Figure \ref{fig:mri_freq70} shows that   $50$ iterations of ADMM with $\gamma = \frac{1}{22}$  produce a result satisfactory to the human eye.  Table \ref{tab:gpuVsError} shows 
the computational time on GPU, and the reference $u_*$ is the numerical solution after $5000$ ADMM iterations.

 Finally, we consider single precision computation on GPU, which is sufficient for many imaging purposes. 
 Results in \cite{liuGPU} show that single precision computation allows computation of larger problems on one GPU card due to the consumption of less memory. The Python package JAX offers two options for single-precision computing with default Float-32 (FP32), and also TensorFloat-32 (TF32), see \cite{liuGPU} for technical details. These tests were conducted for 3D  Shepp-Logan images with $30\%$ observed frequencies, and ADMM with  $\gamma = \frac{1}{\tau} = \frac{1}{22}$. Figure \ref{fig:precision_comparison} shows that single-precision computation does not affect the local linear rate. 
In Table \ref{tab:example}, we see that single-precision computing is not only faster than double-precision (FP64), but it also allows us to compute problems of size $700^3$ while double-precision runs out of memory for any problem larger than $512^3$ on one Nvidia A100 80GB card. Moreover, the difference in speed between double-precision and single-precision is widened as the size of the problem grows larger.

\section{Concluding Remarks}
\label{sec:remark}
In this paper, we have provided an asymptotic linear convergence rate of ADMM applied to the Total Variation Compressed Sensing (TVCS) problem by applying DRS to an equivalent problem.  The explicit rate shows the similarities and differences between TVCS and Basis Pursuit. The results were validated with large 3D tests, where a simple but efficient GPU implementation was provided. Among these results, it was shown that the generalized version of ADMM on the regularized TVCS problem has the potential to speed up the convergence rate as in Basis Pursuit. This intuition could shed some light on how to choose parameters for the TVCS problem as well.
{Future work includes possible extensions to regularized problems.}

\begin{acknowledgements}
E.T. and X.Z. were supported by NSF grant DMS-2208515, and  M.J. was supported by NSF grant  DMS-2400641.\end{acknowledgements}
\begin{datavail}
    MRI data was provided with the consent of the individual(s). The source wished to remain unacknowledged, and all data handling complied with applicable privacy and ethical guidelines to maintain confidentiality and respect the source's wishes for anonymity. All other datasets generated during and/or analysed during the current study are available from the authors on reasonable request.
\end{datavail}

\section*{Appendix}
\appendix 
\section{Derivation of Dual Problems}
\label{appendix-dual}

 For   $G^*=g^*\circ (-\mathcal K^*)$ where $g=\iota_{\{u\in \R^N: Au=b\}}(u)$ as defined in Section \ref{sec-notation-constraints}, we derive its convex dual function $G=(G^*)^*$.
 Let $\mathrm{P}(v)$ denote the projection of $v\in \R^N$ onto the affine set $\{u\in \R^N: Au=b\}$.
 Recall that $A\in \mathbb C^{m\times N}$ defined in  Section \ref{sec-notation-constraints} satisfies $AA^*=I$, thus
$\mathrm{P}(v)=v+A^*(AA^*)^{-1}(b-Av) = v + A^*(b-Av).$
For any $p \in \Im(A^*)$, let $p = A^*z$, then $z = Ap$ due to the fact $AA^* = I_{m\times m}$.
The convex conjugate of $g=\iota_{\{u\in \R^N: Au=b\}}(u)$ is the support function of the affine set, which can be simplified as follows. For $p \in \Im(A^*)$, 
\begin{equation*}
   g^*(p) = \sup_{u: Au=b}\langle u, p\rangle=\sup_{u: Au=b}\langle u, A^*z\rangle =\sup_{u: Au=b}\langle Au, Ap\rangle = \langle b, Ap \rangle= \langle A^*b, p \rangle.
\end{equation*}
Thus
$ g^*(p) = \begin{cases}
                \langle p, A^* b \rangle_{\R^N}, \quad &\mathrm{if} \ p \in \Im(A^*) \\
                +\infty, & \mathrm{otherwise}
    \end{cases}$.

Let $(\mathcal K^*)^{-1}\big[\Im (A^*)\big]$ be the pre-image of $\Im (A^*)$ under $\mathcal K^*$.
By the derivation above, 
\begin{align*}
    g^*(-\mathcal K^*p) &= \begin{cases}
                -\langle \mathcal K^*p, A^* b \rangle \quad &\mathrm{if} \ \mathcal K^*p \in \Im(A^*) \\
                +\infty & \mathrm{otherwise}.
    \end{cases} =  -\langle p, \mathcal KA^* b \rangle + \iota_{ (\mathcal K^*)^{-1}\big[\Im (A^*)\big]}(p).
\end{align*}
  Notice that  $\langle p, \mathcal KA^* b \rangle$ is continuous in $p$. Since $\Im (A^*)$ is a closed set and $\mathcal K^*$ is a  bounded linear transformation, $(\mathcal K^*)^{-1}\big[\Im (A^*)\big]$ is a closed convex set. Since an indicator function of a closed convex set is a closed convex proper function, $G^* = g^*\circ (-\mathcal K^*)$ is a closed convex proper function. By the  regularity condition $\Im(A^*) \cap \Im (\mathcal K^*) \ne \emptyset$, we have 
\begin{align*}
    & -\mathcal K\{u:Au=b\}=-\mathcal K\pd g^*(-\mathcal K^*p) = \pd(g^*\circ (-\mathcal K^*))(p) = \pd G^*(p), \\
    & \qquad \forall p \in (\mathcal K^*)^{-1}\big[\Im (A^*)\big].
\end{align*}
So  
    $G(v)=[G^*]^*(v) =  \sup_{p}[ \langle v, p\rangle_{\R^{2N}} - G^*(p)] = \begin{cases}
        0 \quad &\mathrm{if} \ v \in -\mathcal K\{u:Au = b\} \\
        +\infty &\mathrm{otherwise}
    \end{cases}.$
Define $h(v):= G(-v) =\iota_{\mathcal K\{u:Au = b  \}}(v)$, then we have derived the formulation \eqref{eq:doubleDual}. 
Since $G(v)$ is an indicator function, its proximal operator is a projection, which can be written as
    $$\prox_{G^*}^{\gamma}(q) = \mathcal{F}^*  \Big[\widetilde{M}^*\widetilde{M}+ (I-\widetilde{M}^* \widetilde{M}) (I- \Lambda (\Lambda^* \Lambda)^+ \Lambda^*)\Big] \mathcal{F}(q+\gamma \mathcal KA^*b),$$
    where $\widetilde{M}$ is defined in \eqref{M-tilde}.
Then by Theorem \ref{eqn:pd_relation} (i) and (v), we obtain  $\prox_h^\tau$  as \eqref{prox-h}.

\section{Equivalence of DRS on Primal and Dual Problems}\label{appendix-drs-primal-dual}
Consider $ (P)   \min_x f(x) + g(x)  $ and $
 (D)    \min_p f^*(p) + g^*(-p) $ for two closed convex proper functions $f(x)$ and $g(x)$.
DRS with a step size $\gamma>0$ and a relaxation parameter $\lambda$ for (P) is
\begin{equation}
\label{DR-primal}
\begin{cases}
      s_{k+1} &= s_k - \lambda t_k + \lambda  \prox^{\gamma}_g(2t_k-s_k),\quad \lambda\in (0,2) \\
    t_{k} &= \prox_f^{\gamma}(s_{k})
\end{cases}.
\end{equation}
With the fact $\prox_{g^*\circ(-\iden)}^\tau(p)=-\prox_{g^*}^\tau(-p)$,  DRS with a step size $\frac{1}{\gamma}$  and a relaxation parameter $\lambda$ for the dual problem can be written as
\begin{equation}
\label{DR-dual}
\begin{cases}
    q_{k+1} &= q_k -  \lambda p_k - \lambda  \prox_{g^*}^{\frac{1}{\gamma}}(-2p_k+q_k),\quad \lambda\in (0,2) \\
    p_{k} &= \prox_{f^*}^{\frac{1}{\gamma}} (q_{k})
    \end{cases}.
\end{equation}
With
 Moreau Decomposition,  
\eqref{DR-primal} is equivalent to \eqref{DR-dual}
via 
$    q_k = \frac{s_k}{\gamma},    p_k = \frac{s_k - t_k}{\gamma}.$

\section{Proof of Equivalence of G-prox PDHG and ADMM}\label{sec:equiv}
 \label{appendix-GroxPDHG}
 We give the proof of  Theorem \ref{thm-equivalence-GproxPDHG}.
 The main tool we will need is the following lemma:
\begin{lemma}\label{theo:chainRule}
For a closed convex proper function $h$, $\beta >0$, and a matrix $\mathcal K$,
\[ \hat{p} = \argmin_p h(p)+\frac{\beta}{2}\|\mathcal Kp-q\|^2 \implies \beta(\mathcal K\hat{p}-q) = \prox^{\beta}_{h^*\circ (-\mathcal K^*)}(-\beta q).\]
\end{lemma}
\begin{proof} 
By  Theorem \ref{eqn:pd_relation} (iii), we have $0 \in \pd h(\hat{p}) + \beta \mathcal K^*(\mathcal K\hat{p} -q)$, which holds if and only if $\hat{p} \in \pd h^* \Big( -\beta \mathcal K^*(\mathcal K\hat{p}-q) \Big)$.
Multiplying both sides  by $-\mathcal K$, we get
$-\mathcal K\hat{p} \in -\mathcal K \pd h^*\Big[ -\beta \mathcal K^* (\mathcal K\hat{p}-q) \Big].$
Let $y= \beta \Big( \mathcal K\hat{p} - q \Big)$ and $g(x) = -\mathcal K^*x$. By chain rule, we have $$-\mathcal K\pd h^*\Big[g(y) \Big] = \pd (h^*\circ g)(y) = \pd [h^* \circ (-\mathcal K^*)] \Big( \beta[\mathcal K\hat{p}-q] \Big)\Rightarrow -\mathcal K\hat{p} \in \pd [h^* \circ (-\mathcal K^*)]\Big(\beta[\mathcal K\hat{p}-q]  \Big).$$
By adding $\mathcal K\hat{p}-q$ then multiplying $\beta$ to both sides, we get
	\begin{align*}
		-\beta q \in \beta(\mathcal K\hat{p}-q) + \beta \pd(h^* \circ -\mathcal K^*)\Big[\beta(\mathcal K\hat{p}-q) \Big] = \Big[ I + \beta \pd [h^* \circ (-\mathcal K^*)] \Big] \Big(\beta(\mathcal K\hat{p}-q)  \Big),
	\end{align*}
which implies  $\beta(\mathcal K\hat{p}-q) = \Big[ I + \beta \pd (h^* \circ -\mathcal K^*)  \Big]^{-1}(-\beta q) = \prox^{\beta}_{h^*\circ -\mathcal K^*}(-\beta q).$  
\qed
\end{proof}	
The first line of G-prox PDHG with step size $\tau$ in Algorithm \ref{alg-Gprox-PDHG}  can be written as $ u_{k+1} = \argmin_u g(u) +   \frac{1}{2\tau}\|\mathcal Ku-(\mathcal Ku_k-\tau w_k)\|^2.$
Applying Lemma \ref{theo:chainRule} to the line above with $h = g$, $\widehat{p}=u_{k+1}$, $\beta = \frac{1}{\tau}$, and $q = \mathcal Ku_k - \tau w_k$, we get
$ \mathcal Ku_{k+1} -\mathcal Ku_k + \tau w_k = \tau \prox^{\frac{1}{\tau}}_{g^*\circ(-\mathcal K^*)}(w_k - \frac{1}{\tau} \mathcal Ku_k  ).$
 By Moreau  Decomposition, the second line of G-prox PDHG with $\tau = \frac{1}{\sigma}$ can be written as
\begin{align*}
    v_{k+1} &=  \argmin_v f^*(v)  + \frac{\tau}{2} \|v -(v_k+\frac{1}{\tau}\mathcal Ku_{k+1})\|^2 = v_k + \frac{1}{\tau}\mathcal Ku_{k+1} - \frac{1}{\tau}\prox_f^{\tau}(\tau v_k + \mathcal Ku_{k+1}).
\end{align*}
Thus, the G-prox PDHG in Algorithm \ref{alg-Gprox-PDHG}  with $\tau = \frac{1}{\sigma}$ yields
\begin{subequations}
    \label{gprox-form}
    \begin{align}
		\mathcal Ku_{k+1} -\mathcal Ku_k + \tau w_k &= \tau \prox^{\frac{1}{\tau}}_{g^*\circ(-\mathcal K^*)}(w_k - \frac{1}{\tau}\mathcal  Ku_k)     \label{gprox-form-1} \\
		\tau v_{k+1} &= \tau v_k + \mathcal Ku_{k+1} - \prox_f^{\tau}(\tau v_k + \mathcal Ku_{k+1}) \label{gprox-form-2} \\
		w_{k+1} &= 2v_{k+1} - v_k. \label{gprox-form-3}
\end{align}
\end{subequations}
The first line in Algorithm \ref{alg-ADMM} 
can be written as $ x_{k+1}= \argmin_x \ g(x) + \frac{\gamma}{2}\|\mathcal Kx-(y_k-\frac{1}{\gamma}z_k)\|^2.$
By Lemma \ref{theo:chainRule} with $h = g$, $\beta = \gamma$, $\widehat{p} = x_{k+1}$, and $q = y_k - \frac{1}{\gamma}z_k$, we get
\[ \resizebox{0.99\textwidth}{!}{$ -\gamma \mathcal Kx_{k+1} - (z_k - \gamma y_k) = \prox_{g^*\circ \mathcal K^*}^{\gamma}\big[\gamma y_k - z_k \big] 
   \Longleftrightarrow \gamma \mathcal Kx_{k+1} +(z_k - \gamma y_k)= \prox_{g^*\circ (-\mathcal K^*)}^{\gamma}\big[ z_k-\gamma y_k \big].$}\]
By the definition of the proximal operator, the second line in Algorithm \ref{alg-ADMM} reduces to
\[y_{k+1} = \argmin_y f(y) - \langle y,z_k\rangle + \frac{\gamma}{2} \|y -\mathcal Kx_{k+1}\|^2  = \prox_{f}^{\frac{1}{\gamma}}\Big[ \frac{1}{\gamma} (z_k + \gamma \mathcal K x_{k+1})\Big].\]
Thus  the ADMM in Algorithm \ref{alg-ADMM}  is equivalent to
\begin{subequations}
    \label{ADMM-form}
\begin{align} 
		\gamma \mathcal Kx_{k+1} +(z_k - \gamma y_k) &= \prox_{g^*\circ -\mathcal K^*}^{\gamma}\big[ z_k-\gamma y_k \big] \label{ADMM-form-1} \\
		y_{k+1} &= \prox_{f}^{\frac{1}{\gamma}}\Big[ \frac{1}{\gamma} (z_k + \gamma \mathcal K x_{k+1})\Big] \label{ADMM-form-2} \\
		z_{k+1} &= z_k - \gamma(y_{k+1}-\mathcal Kx_{k+1}). \label{ADMM-form-3}
\end{align}    
\end{subequations}

Finally, we prove the equivalence between \eqref{gprox-form} and \eqref{ADMM-form}. Define the following variables,
\begin{equation*}
    \tau: = \frac{1}{\gamma},\quad
	v_k: = z_k, \quad
	u_k: = x_k, \quad
	\tau w_k: =\mathcal Kx_k +\tau z_k - y_k,
\end{equation*}
 then \eqref{ADMM-form-1} becomes   \eqref{gprox-form-1} by
\begin{align*}
	&\frac{1}{\tau} \mathcal Kx_{k+1} +(z_k - \frac{1}{\tau} y_k) = \prox_{g^*\circ (-\mathcal K^*)}^{\frac{1}{\tau}}\big[ z_k-\frac{1}{\tau} y_k \big] \\
	\iff &\mathcal Kx_{k+1} +\tau z_k - y_k = \tau \prox_{g^*\circ (-\mathcal K^*)}^{\frac{1}{\tau}}\big[ z_k-\frac{1}{\tau} y_k \big] \\
	 \iff& \mathcal Ku_{k+1} +\tau v_k - \big[\mathcal Ku_k + \tau (v_k - w_k) \big] = \tau \prox^{\frac{1}{\tau}}_{g^*\circ (-\mathcal K^*)}\Big[v_k - \frac{1}{\tau}\big(\mathcal Ku_k + \tau (v_k - w_k) \big) \Big] \\
	 \iff& \mathcal K(u_{k+1} - u_k) + \tau w_k = \tau \prox^{\frac{1}{\tau}}_{g^*\circ (-\mathcal K^*)} \Big[w_k - \frac{1}{\tau}\mathcal Ku_k \Big],
\end{align*}
  \eqref{ADMM-form-2} becomes   \eqref{gprox-form-2} by
\begin{align*}
	y_{k+1} = \prox^{\frac{1}{\gamma}}_f\Big[ \frac{1}{\gamma}(z_k + \gamma \mathcal Kx_{k+1}) \Big]
	&\iff \mathcal Ku_{k+1} + \tau(v_{k+1} -w_{k+1}) = \prox^{\tau}_f\Big[\tau v_k + \mathcal Ku_{k+1} \Big],
\end{align*}
and  \eqref{ADMM-form-3} becomes   \eqref{gprox-form-3} by
\begin{align*}
    z_{k+1} = z_k - \frac{1}{\tau}(y_{k+1}-\mathcal Kx_{k+1}) &\iff 
v_{k+1} = v_k + (w_{k+1}-v_{k+1}) \iff w_{k+1}=2v_{k+1}-v_k.
\end{align*}

\section{Derivation of the Explicit Implementation Formula}
\label{appendix-implementation}

With $  g(u)=\iota_{\{u\in \R^N: \widehat{u}(\ell) = b_{\ell}, \ell\in S\}}(u), \quad f^*(v)=\iota_{\{v\in [\R^N]^d:\|v\|_{\infty,2} \le 1\}}(v),$
we reformulate
\eqref{eq:tvcs_primal-1} into the form of
\eqref{eq:primala},
then we apply G-prox PDHG to \eqref{eq:primala} to obtain
\begin{subequations}
    \begin{align} 
        u_{k+1} &= \argmin_{\{u\in \R^N: \widehat{u}(\ell) = b_{\ell}, \ell\in S \}} \langle \mathcal K u , w_k\rangle + \frac{1}{2\tau}\|\mathcal K(u-u_k)\|^2  \label{explict-1}\\
        v_{k+1} &= \argmax_{ \{v\in [\R^N]^d : \|v\|_{\infty,2} \le 1\} } \langle \mathcal K u_{k+1} , v\rangle - \frac{1}{2\sigma} \|v-v_k\|^2  \label{explict-2} \\
        w_{k+1} &= 2v_{k+1} - v_k.
\end{align}
\end{subequations}

From now on, we focus on the two-dimensional problem and the extension to higher dimensions is straightforward.
We first derive an explicit formula of \eqref{explict-1}. With the notation in Section \ref{sec-notation}, let $\mathcal F u$ and $\widehat u$ be the normalized discrete Fourier transform, i.e., $\mathcal F u=\widehat u$ and $\langle u, v  \rangle_{\R^N}=\langle \mathcal F u,\mathcal F v  \rangle_{\mathbb C^N}$. 
Notice that the matrix $K$ is circulant thus diagonalizable by the 1D normalized DFT
matrix $T$. 
Regard $\mathcal F$ as an $N\times N$ matrix, then with \eqref{discreteGrad-Fourier}, we get
\begin{align*} 
&  \argmin_{\{u\in \R^N: \widehat{u}(\ell) = b_{\ell}, \ell\in S \}} \langle \mathcal K u , w_k\rangle_{\R^{2N}} + \frac{1}{2\tau}\|\mathcal K(u-u_k)\|_{\R^{2N}}^2  \label{explict-1}\\
    =& \argmin_{\{u: \widehat{u}(\ell) = b_{\ell}, \ell\in S \}}  \langle  \begin{pmatrix}
        \mathcal F & 0 \\
        0 & \mathcal F
    \end{pmatrix} \mathcal K u, \begin{pmatrix}
        \mathcal F & 0 \\
        0 & \mathcal F
    \end{pmatrix} w_k \rangle_{\mathbb C^{2N}} + \frac{1}{2\tau }\|\begin{pmatrix}
        \mathcal F & 0 \\
        0 & \mathcal F
    \end{pmatrix} \mathcal K(u-u_k)\|_{\mathbb C^{2N}}^2 \\
    =& \argmin_{\{u: \widehat{u}(\ell) = b_{\ell}, \ell\in S \}} \langle  u, \mathcal F^* \BLambda^*  \begin{pmatrix}
        \mathcal F & 0 \\
        0 & \mathcal F
    \end{pmatrix} w_k \rangle_{\mathbb C^{2N}} + \frac{1}{2\tau}\|\BLambda  \mathcal F(u-u_k)\|_{\mathbb C^{2N}}^2 .
\end{align*}
Let $\bar v$ denote the complex conjugate of $v$.
Since both $\mathcal F^*\BLambda^*\begin{pmatrix}
        \mathcal F & 0 \\
        0 & \mathcal F
    \end{pmatrix}=\mathcal K^*$ and $\mathcal F^*\BLambda^*\BLambda \mathcal F=\mathcal K^*\mathcal K$ are  real-valued matrices,
by taking the derivative with respect to   $u\in \R^N$, we get 
$\tau  \mathcal F^*\BLambda^*\begin{pmatrix}
        \mathcal F & 0 \\
        0 & \mathcal F
    \end{pmatrix}w_k+\mathcal F^*\BLambda^*\BLambda \mathcal F(u_{k+1}-u_k)=0$.
 For   $w\in \R^{2N}$, let $w=\begin{pmatrix}
     w^1 \\ w^2
 \end{pmatrix}$ with $w^1, w^2\in \R^N$.
  With the notation $\mathcal F w=\widehat w$, 
  we have 
  $\mathcal F^*\BLambda^*\begin{pmatrix}
        \mathcal F & 0 \\
        0 & \mathcal F
    \end{pmatrix}w=\mathcal F^*\BLambda^* \begin{pmatrix}
     \widehat w^1 \\ \widehat w^2
 \end{pmatrix} =\mathcal F^*  \begin{pmatrix}
    (\Lambda\otimes I) \widehat w^1 \\ (I\otimes \Lambda)\widehat w^2
 \end{pmatrix}.$
  Let $\lambda^1_{\ell}$ $(\ell=1,\cdots, N)$ be the diagonal entries of $\Lambda\otimes I$
  and $\lambda^2_{\ell}$ $(\ell=1,\cdots, N)$ be the diagonal entries of $I\otimes \Lambda$,  then the update rule in the Fourier domain yields
\begin{align*}
       \widehat{u_{k+1}}(\ell) &= b_{\ell}, \ & \ \ell \in  S \\
\widehat{u_{k+1}}(\ell) &=        \widehat{u_k}(\ell) - \tau \frac{\overline{\lambda^1_{\ell}}\widehat{w^1_k}(\ell)+\overline{\lambda^2_{\ell}}\widehat{w^2_k}(\ell)}{|\lambda_{\ell}^1|^2+|\lambda_{\ell}^2|^2}, \ & \ell \notin S.
\end{align*}

Since \eqref{explict-2} can be rewritten as
$ v_{k+1} = \argmin_{\{v\in [\R^N]^d : \|v\|_{\infty,2} \le 1\} } \|v-(v_k+\sigma \mathcal K u_{k+1})\|^2,$  
\eqref{explict-2} can be implemented as the projection of $v_{k} + \sigma \mathcal K u_{k+1}$
onto the $\|\cdot\|_{\infty,2}$ ball: 
\begin{align*}
    v_{k+1} = \mathrm{Projection}_{\{v: \|v\|_{\infty,2}\leq 1 \}}(v_{k} + \sigma \mathcal K u_{k+1})= \frac{v_{k} + \sigma \mathcal K u_{k+1}}{\max (1, |v_{k} + \sigma \mathcal K u_{k+1}|)}. 
\end{align*}

\end{document}